\documentclass[reqno]{amsart}
\usepackage{geometry}
\usepackage[utf8]{inputenc}
\usepackage{esint} %
\usepackage{stmaryrd} %
\SetSymbolFont{stmry}{bold}{U}{stmry}{m}{n}  %
\usepackage{xfrac} %
\usepackage{comment}

\usepackage{mathtools,microtype,lmodern,amssymb}
\mathtoolsset{centercolon}
\usepackage[dvipsnames]{xcolor}
\usepackage{enumerate}
\usepackage{tikz}
\usetikzlibrary {arrows.meta} 
\usepackage[colorlinks=true,citecolor=Cyan]{hyperref}

\DeclarePairedDelimiter{\abs}{\lvert}{\rvert}

\DeclarePairedDelimiter{\bra}{(}{)}
\DeclarePairedDelimiter{\pra}{[}{]}
\DeclarePairedDelimiter{\set}{\{}{\}}

\DeclarePairedDelimiter{\ppra}{\llbracket}{\rrbracket}

\newcommand{\Z}{\mathbb{Z}}
\newcommand{\R}{\mathbb{R}}
\newcommand{\Rp}{\mathbb{R}_{>0}}
\newcommand{\Rnn}{\mathbb{R}_{\ge0}}

\newcommand{\dst}{{\mathbb{D}}}

\newcommand{\Prob}{\mathcal{P}}

\newcommand{\dprbtwo}{\mathcal{P}_{2,c}^\delta}
\newcommand{\delltwo}{\ell^1_{2,c}}

\newcommand{\intr}{\int_{\mathbb R}}
\newcommand{\dn}{\mathrm{d}}
\newcommand{\dd}{\,\mathrm{d}}
\DeclareMathAlphabet{\mathup}{OT1}{\familydefault}{m}{n}
\newcommand{\dx}[1]{\mathop{}\!\mathup{d} #1} %

\newcommand{\px}{\partial_x}
\newcommand{\pxx}{\partial_{xx}}

\newcommand{\fnc}{\mathcal{E}}

\newcommand{\eps}{\varepsilon}
\newcommand{\bdry}[3][K]{\ppra*{#2,#3}_{#1}}

\newcommand{\leqcx}{\leq_{\mathsf{cx}}}

\newcommand{\sgrp}{{\mathfrak S}}

\newcommand{\frc}[2]{\pra*{\frac{#1^2}{#2}}}

\newcommand{\Mart}{\textup{Mart}}

\newcommand{\ddff}{\Delta^{\!\delta}}

\newcommand{\dcent}[1]{\mathbf{m}_1^\delta(#1)}
\newcommand{\dmom}[1]{\mathbf{m}_2^\delta(#1)}

\newcommand{\hell}{\mathbb{H}}
\newcommand{\dhell}{\mathbb{H}^\delta}
\newcommand{\wass}{\mathbb{W}}

\newcommand{\welo}{\mathfrak{w}}
\newcommand{\curves}{{\mathcal C}}

\newcommand{\daction}{\mathfrak{A}^\delta}

\newcommand{\tnorm}[1]{{\left\vert\kern-0.25ex\left\vert\kern-0.25ex\left\vert #1 
    \right\vert\kern-0.25ex\right\vert\kern-0.25ex\right\vert}}

\newtheorem{theorem}{Theorem}
\newtheorem{proposition}[theorem]{Proposition}
\newtheorem{lemma}[theorem]{Lemma}
\newtheorem{corollary}[theorem]{Corollary}
\newtheorem{example}[theorem]{Example}
\newtheorem{definition}[theorem]{Definition}
\newtheorem{remark}[theorem]{Remark}

\date{\today}

\title[Diffusive transport and semi-contractive flows]{Diffusive transport on the real line: semi-contractive gradient flows and their discretization}
\author{Daniel Matthes}
\address{Daniel Matthes -- Zentrum Mathematik, TU München, Boltzmannstrasse 3, D-85748 Garching, Germany}
\email{matthes@ma.tum.de}
\author{Eva-Maria Rott}
\address {Eva-Maria Rott -- Zentrum Mathematik, TU München, Boltzmannstrasse 3, D-85748 Garching, Germany}
\email{eva-maria.rott@tum.de}
\author{André Schlichting}
\address{André Schlichting -- Institute for Applied Analysis, Ulm University, Helmholtzstrasse 18, D-89081 Ulm, Germany}
\email{andre.schlichting@uni-ulm.de}

\allowdisplaybreaks

\begin{document}

\begin{abstract}
  The diffusive transport distance, a novel pseudo-metric between probability measures on the real line, is introduced. It generalizes Martingale optimal transport, and forms a hierarchy with the Hellinger and the Wasserstein metrics.
  We observe that certain classes of parabolic PDEs, among them the porous medium equation of exponent two, are formally semi-contractive metric gradient flows in the new distance.
  This observation is made rigorous for a suitable spatial discretization of the considered PDEs: these are semi-contractive gradient flows with respect to an adapted diffusive transport distance for measures on the point lattice. The main result is that the modulus of convexity is uniform with respect to the lattice spacing. Particularly for the quadratic porous medium equation, this is in contrast to what has been observed for discretizations of the Wasserstein gradient flow structure.
\end{abstract}

\thanks{\emph{Acknowledgement:} The authors thank Giuseppe Savaré for sharing his ideas about the introduction of the diffusive transport metric, and Martin Huesmann for many discussions about possible relations to Martingale transport. \\
	\emph{Funding:}  DM's and ER's research is supported by the DFG Collaborative Research Center TRR 109, ``Discretization in Geometry and Dynamics.''
	AS is supported by the Deutsche Forschungsgemeinschaft (DFG, German Research Foundation) under Germany's Excellence Strategy EXC 2044 --390685587, Mathematics M\"unster: Dynamics--Geometry--Structure.}


\maketitle

\section{Introduction}
\subsection{Motivation and outline}
Evolution equations of the form
\begin{equation}\label{eq:gfformal}
 \partial_t\rho = -\pxx\big(\rho\,\pxx\fnc'(\rho)\big), 
\end{equation}
where $\fnc$ is a functional on the space of probability densities, carry the formal structure of a metric gradient flow \cite{AGS}. In the paper at hand, we introduce the corresponding (pseudo-)metric, which we call diffusive transport distance, identify families of semi-contractive gradient flows of the form \eqref{eq:gfformal}, and perform a spatial discretization of the latter that preserves the contractivity.

PDEs of type \eqref{eq:gfformal} can be considered as fourth order analogues of gradient flows in the $L^2$-Wasserstein metric, which have the form
\begin{equation}\label{eq:gfW}
    \partial_t\rho = \px\big(\rho\,\px\fnc'(\rho)\big).
\end{equation}
Famous examples that fit \eqref{eq:gfW}, first identified in \cite{JKO1998,OttoPME}, are the linear Fokker-Planck equation $\partial_t\rho=\pxx\rho+\px(\rho\px V)$ and the quadratic porous medium equation $\partial_t\rho=\pxx\rho^2$, obtained for the logarithmic relative entropy $\fnc(\rho)=\int \rho\log(\rho/e^{-V})\dd x$ with a convex potential $V$, and the quadratic Renyi entropy $\fnc(\rho)=\int\rho^2\dd x$, respectively. These two flows also belong to the rare collection of \emph{contractive} gradient flows in the $L^2$-Wasserstein metric.

The primary example of a gradient flow of the form \eqref{eq:gfformal} is the fourth order DLSS equation \cite{DerridaLebowitzSpeerSpohn1991a,DerridaLebowitzSpeerSpohn1991b,Bleher,JunPin-analysis,GianazzaSavareToscani2009}
\begin{align}
  \label{eq:DLSS}
  \partial_t\rho = -\pxx\big(\rho\,\pxx\log\rho\big),
\end{align}
obtained for the logarithmic entropy (with $V\equiv0$). In our recent paper \cite{MRSS}, we have used that this particular gradient flow structure to design an adapted spatial discretization of \eqref{eq:DLSS} whose solutions share a surprising amount of qualitative properties with solutions to \eqref{eq:DLSS}. 
Further applications of the gradient flow structure \eqref{eq:gfformal}, or rather variants with more sophisticated mobilities than $\rho$, include models for wealth exchange from econophysics~\cite{Cohen2024} and fourth order corrections to the heat equation~\cite{PanXuLouYao2006,Nika2023}. 

To the best of our knowledge, contractivity of the gradient flows \eqref{eq:gfformal} has not been studied before. 
Thus, the evolution equations given below in \eqref{eq:Fpotential}--\eqref{eq:Fpme} --- among them the linear heat and the quadratic porous medium equations --- are apparently the first examples of semi-contractive flows in the diffusive transport distance. 
\medskip

The paper is divided into two parts. In the first part, we collect properties of the diffusion transport distance and the gradient flows \eqref{eq:gfformal}, particularly about the semi-contractivity of the specific examples \eqref{eq:Fpotential}--\eqref{eq:Fpme}. The second part is about the spatially discrete analogue of diffusive transport and the spatial discretization of the flows \eqref{eq:gfformal}. Our main contribution is the proof of mesh-uniform preservation of the contractivity rates for the aforementioned examples \eqref{eq:Fpotential}--\eqref{eq:Fpme}. The results in the second part, which are the more difficult and even more suprising ones, are proven in full rigor. In the first part, which is mainly intended as an overview, we remain mostly formal; most of the results there can be proven either by standard methods, or by passing to the limit in the discretization.

\subsection{Hellinger, Kantorovich and diffusive metrics}
The Hellinger distance $\hell$ and the Monge--Kantorovich --- or $L^2$-Wasserstein --- distance $\wass$ both define metrics on the space $\Prob_2(\R)$ of probability measures with finite second moment on the real line $\R$. The distances are most easily introduced via their \emph{static formulation},
\begin{align*}
    \hell(\mu_0,\mu_1)^2 = \int \bigl|\sqrt{\mu_0}-\sqrt{\mu_1}\bigr|^2,
    \quad
    \wass(\mu,\nu)^2 = \inf_{\substack{\pi_x\#\gamma={\mu_0}\\ \pi_y\#\gamma={\mu_1}}} \iint |x - y|^2 \dd\gamma(x,y),
\end{align*}
where the second infimum runs over all probability measures $\gamma$ on $\R\times\R$ with respective marginals $\mu$ and $\nu$. 
On the other hand, both $\hell$ and $\wass$ also possess a ``dynamic'' formulation, that is given, respectively, by
\begin{align*}
    \hell(\mu_0,\mu_1)^2 
    &= \inf_{(\mu,u)}\left\{\int_0^1\left(\intr u_s^2\dd\mu_s\right)\dd s \,\middle|\,\partial_s\mu_s=u_s\mu_s\right\}, \\
    \wass(\mu_0,\mu_1)^2 
    &= \inf_{(\mu,v)}\left\{\int_0^1 \left(\intr v_s^2\dd\mu_s\right)\dd s \,\middle|\,\partial_s\mu_s=\partial_x(v_s\mu_s)\right\}.
\end{align*}
Notice that both definitions only differ in the constraint: for $\hell$, the constraint is an ordinary differential equation at every point of $\R$ that determines the temporal rate $u_s$ of creation of annihilation of mass; for $\wass$, it is a continuity equation that determines the velocity field $v_s$ for the transport of mass.

In this paper, following up on \cite{MRSS}, we propose the definition of a further (pseudo-)metric $\dst$, which we call \emph{diffusive transport distance}, with a second order PDE as constraint,
\begin{align}
    \label{eq:dstformal}
    \dst(\mu_0,\mu_1)^2
    &= \inf_{(\mu,w)}\left\{\int_0^1\left(\int w_s^2\dd\mu_s\right)\dd s\,\middle|\,\partial_s\mu_s=\pxx(w_s\mu_s)\right\}.
\end{align}
We emphasize that we do not impose a sign condition on the function $w_s:\R\to \R$, which plays the role of the \emph{diffusivity}. This is in contrast to \emph{Martingale transport}~\cite{HuesmannTrevisan2019}, where $w_s$ is assumed to be non-negative. Thus $\dst$ is weaker than Martingale transport, and genuinely so, since $\dst(\mu_0,\mu_1)$ is finite for certain pairs of measures $\mu_0$ and $\mu_1$ that are not in convex order.

\subsection{Semi-contractive gradient flows}
\label{sct:intro.formal}
Our primary motivation for introducing $\dst$ comes from gradient flows. %
The metric gradient flow of a functional $\fnc$ with respect to $\dst$ on $\Prob_2(\R)$ yields the PDE \eqref{eq:gfformal}.
Specifically, we consider the following functionals and respective gradient flow PDEs in detail.
\begin{enumerate}
    \item For a given convex external potential $V\in C^2(\R)$, the (negative) potential energy induces a linear diffusion equation of second order: 
    \begin{align}
        \label{eq:Fpotential}
        \fnc(\rho) = -\intr V(x)\dd\rho(x)
        \quad\leadsto\quad
        \partial_t\rho = \pxx (\rho\,\pxx V).
    \end{align}
    \item For a given convex interaction potential $W\in C^2(\R)$, the (negative) interaction energy induces a second order non-local diffusion equation with non-local diffusivity: 
    \begin{align}
        \label{eq:Finteract}
        \fnc(\rho) = - \frac12\intr \rho\ast W\dd\rho
        \quad\leadsto\quad
        \partial_t\rho = \pxx (\rho\,\pxx W\ast \rho).
    \end{align}
    \item A singular limit of the above that is of independent interest is the quadratic porous medium equation, which is obtained for interaction energy $W(z)=\frac14|z|$, i.e.,
    \begin{align}
        \label{eq:Fpme}
        \fnc(\rho) = \frac14\intr\intr |x-y|\rho(x)\rho(y)\dd x\dd y
        \quad\leadsto\quad
        \partial_t\rho = \pxx (\rho^2).
    \end{align}
    Such a connection with martingale optimal transport, was already observed in~\cite[§5.1]{Brenier2020hiddenconvexity}
\end{enumerate}
We can now state our results on semi-contractivity of solutions to \eqref{eq:Fpotential}--\eqref{eq:Fpme} in $\dst$. Recall that a semi-group $\sgrp:\Rnn\times X\to X$ on a metric space $(X,d)$ is $\lambda$-uniformly contractive if
\begin{align}\label{eq:contract}
    d\big(\sgrp^t(x),\sgrp^t(x')\big) \le e^{-\lambda t}d(x,x')
    \quad \text{for all $t\ge0$, $x,x'\in X$}.
\end{align}
\begin{proposition}[Formal geodesic convexity]\label{prp:formal:convexity}
  Formally, we have:
  \begin{enumerate}
  \item \label{itm:potconvexity}
    Assume that $V\in C^4(\R)$ is convex with $V''>0$. Provided that 
    \begin{align}
      \label{eq:potconvexity}
      \Lambda := \sup_\R \frac{\big[(V'')^2\big]''}{V''} < \infty ,
    \end{align}
    the gradient flow in \eqref{eq:Fpotential} is $(-\Lambda)$-uniformly semi-contractive.
  \item \label{itm:intconvexity} Assume that $W\in C^4(\R)$ is convex with $W''>0$. Provided that 
    \begin{align}
      \label{eq:intconvexity}
      \Lambda := 2\sup_\R|W^{IV}| + \sup_\R \frac{\big[(W'')^2\big]''}{W''} < \infty ,
    \end{align}
    the gradient flow in \eqref{eq:Finteract} is $(-\Lambda)$-uniformly semi-contractive.
  \item \label{itm:PMEconvexity} The quadratic porous medium equation \eqref{eq:Fpme} is a contractive gradient flow.
  \end{enumerate}
\end{proposition}
We recall that these results are \emph{formal} since we shall not prove them explicitly. They can be obtained from the corresponding results on $\lambda$-contractivity of the \emph{spatially discretized flows} in Theorem~\ref{thm:contractivity} below
by means of technically cumbersome but by now standard approximation arguments, see e.g. \cite{GigliMaas}. 

The choice of functionals in \eqref{eq:Fpotential} and \eqref{eq:Finteract} is motivated by the theory of $\lambda$-contractive flows in the $L^2$-Wasserstein metric $\wass$, see e.g. \cite[Chapter 9]{AGS}: for the Wasserstein gradient flows \eqref{eq:gfW} of the potential energy and the interaction energy, $\lambda$-contractivity has been proven \cite{CarMcCVil}. In fact, the only other known class of functionals leading to semi-contractive flows in $\wass$ is that of internal energies of the form $\fnc(\rho)=\int f(\rho)\dd x$, where $f$ satisfies McCann's convexity hypothesis \cite{McCann}. But apparently, the respective gradient flows \eqref{eq:gfformal} in $\dst$, including the DLSS equation \eqref{eq:DLSS}, are \emph{not} semi-contractive.

\subsection{Spatial discretization}
In the rigorous part of our paper, we consider a variant of the diffusive transport distance $\dst$ not on probability measures over the real line, but on the one-dimensional lattice $\delta\Z$ of a uniform mesh width $\delta>0$.
Probability densities on $\delta\Z$ are non-negative functions $\rho=(\rho_\kappa)_{\kappa\in\delta\Z}$ such that $\delta\sum_{\kappa\in\delta\Z}\rho_\kappa=1$.

In analogy to \eqref{eq:dstformal}, we introduce for two probability densities $\rho^0$ and $\rho^1$ on $\delta\Z$:
\begin{align}
  \label{eq:intro-dst}
    \dst(\rho^0,\rho^1)^2
    = \inf_{(\rho,\welo)}\set[\bigg]{\int_0^1 \delta \sum_{\kappa\in\delta\Z} \frc{(\welo^s_\kappa)}{\rho^s_\kappa} \dx s  \ \bigg|\ \dot\rho^s = \ddff \welo^s } \,,
\end{align}
where $(\rho^s)_{s\in[0,1]}$ is a curve in the space of probability densities on $\delta\Z$, and $\ddff$ is the discrete Laplacian; see Section \ref{sct:discrete} for the precise definitions.
About $\dst$, we prove:
\begin{theorem}
    $\dst$ is a geodesic pseudo-distance, that is:
    fix a probability density $\rho^*\in \Prob_2(\delta\Z)$, and consider the set $A(\rho^*)$ on which $\dst(\cdot,\rho^*)$ is finite. Then $\dst$ is a metric on $A(\rho^*)$, and for any two $\rho^0,\rho^1\in A(\rho^*)$,
    there exists a minimizer $(\rho^s,\welo^s)_{s\in[0,1]}$ for the problem \eqref{eq:intro-dst} above. 
\end{theorem}
In full analogy to the formal Section \ref{sct:intro.formal} above, we can now list gradient flows in the spatially discrete metric, and the corresponding contractivity results.
Below, $f\ast_\delta g$ denotes the convolution of two grid functions $f,g:\delta\Z\to\R$, see \eqref{eq:dconvolve}.
\begin{enumerate}
    \item For an external potential function $V$ on $\delta\Z$,
    \begin{align}
        \label{eq:dFextern}
        \fnc(\rho) = -\delta \sum_{\kappa\in\delta\Z} V_\kappa\rho_\kappa
        \quad\leadsto\quad
        \dot\rho = \ddff (\rho\,\ddff V).
    \end{align}
    \item For an interaction potential function $W$ on $\delta\Z$,
    \begin{align}
        \label{eq:dFinteract}
        \fnc(\rho) = -\frac\delta2 \sum_{\kappa\in\delta\Z} \rho_\kappa W\ast_\delta\rho
        \quad\leadsto\quad
        \dot\rho = \ddff (\rho\,\ddff W\ast_\delta\rho).
    \end{align}
    \item Concerning the porous medium equation,
    \begin{align}
        \label{eq:dFpme}
        \fnc(\rho) = \frac{\delta^3}4\sum_{\lambda,\kappa\in\delta\Z} |\lambda-\kappa|\rho_\kappa\rho_\lambda
        \quad\leadsto\quad
        \dot\rho = \ddff(\rho^2).        
    \end{align}
\end{enumerate}
Our analytical results are as follows.
\begin{theorem}
\label{thm:contractivity}
    For the spatially discrete flows given above, we have:
    \begin{enumerate}
    \item \label{lbl:contractV}
        Assume that $\ddff V^\delta$ is positive and bounded. Provided that
        \begin{align}
           \label{eq:lambda.1}
            \Lambda = \sup \frac{\ddff\big[\big(\ddff V^\delta\big)^2\big]}{\ddff V^\delta}
        \end{align}
        is finite, then the gradient flow \eqref{eq:dFextern} is $(-\Lambda)$-contractive.
    \item \label{lbl:contractW}
        Assume that $\ddff W^\delta$ is positive and bounded. Provided that
        \begin{align}
            \label{eq:lambda.2}
            \Lambda = 2\sup|\ddff\ddff W^\delta| 
            + \sup\frac{\ddff\big[\big(\ddff W^\delta\big)^2\big]}{\ddff W^\delta}
        \end{align}
        is finite, then the gradient flow \eqref{eq:dFinteract} for $\fnc$ is $(-\Lambda)$-contractive.
    \item \label{lbl:contractPME}
        The discrete quadratic porous medium equation \eqref{eq:dFpme} is contractive.
    \end{enumerate}
\end{theorem}
These results are by no means obvious, and even given the contractivity rates in Proposition \ref{prp:formal:convexity} for the examples \eqref{eq:Fpotential}--\eqref{eq:Fpme}, there is a priori no reason to expect the same rates for their discretizations \eqref{eq:dFextern}--\eqref{eq:dFpme}.
It is a general observation that uniform geodesic semi-convexity is fragile with respect to spatial discretization. A comparable situation is that of the linear Fokker-Planck and the quadratic porous medium equation as gradient flows in the $L^2$-Wasserstein metric $\wass$: it has been shown in \cite{MielkeMC} that contractivity of the linear Fokker-Planck equation is preserved for a very carefully designed discretization of the gradient flow structure, see also \cite{Maas2011,Chow,ErbarMaas2012,EFS19}, or \cite{MaasMatthes} for a generalization to a nonlinear fourth order equation; in contrast, for discretization of the quadratic porous medium equation from \eqref{eq:Fpme} with an analogous approach, the modulus of $\lambda$-contractivity diverges to $-\infty$ as the lattice spacing vanishes \cite{ErbarMaas}.

\subsection{Outlook} %
We finish with giving two directions of future research:
\begin{enumerate}
    \item \emph{Combining diffusive transport with $L^2$-Wasserstein:} Gradient flows in such a combined metric will take the form, with some parameter $\eps>0$,
    \begin{align*}
    \partial_t\rho = -\pxx\big(\rho\,\pxx\fnc'(\rho)\big) + \eps\,\partial_x\big(\rho\,\partial_x\fnc'(\rho)\big).
    \end{align*}
    Choosing in particular for $\fnc$ a relative entropy for the square potential,
    \begin{align}\label{eq:relent}
        \fnc(\rho) =\intr \left(\rho\log\rho + \frac\eps 2|x|^2\rho\right)\dd x,
    \end{align}
    we obtain the following re-scaled DLSS equation,
    \begin{align}\label{eq:weirddlss}
        \partial_t\rho
        = -\pxx\big(\rho\,(\eps+\pxx\log\rho)\big) + \eps\,\partial_x\big(\rho(\eps x+\partial_x\log\rho)\big)
        = -\pxx(\rho\,\pxx\log\rho) + \eps^2\,\partial_x(x\rho) \,.
    \end{align}
    A natural question is if the uniform geodesic convexity of this functional \eqref{eq:relent} in plain $L^2$-Wasserstein $\wass$ suffices to make \eqref{eq:weirddlss} contractive in the combined metric. 
    \item \emph{Generalization to other mobilities:} Having in mind e.g. applications in econophysics \cite{Cohen2024}, alternative definitions of the diffusive distance, with mobilities different from the linear one, can be considered, leading to a different notion of contractivity. Actually, it is already interesting to study the effect of alternative discretizations for the linear mobility: for our discretization of the DLSS equation \eqref{eq:DLSS} in \cite{MRSS}, the choice of the discrete mobility --- which is much more sophisticated than a point evaluation of the density --- has been of crucial importance for structure preservation.

\end{enumerate}

\subsection{Plan of the paper}
In the formal part, Section \ref{sct:formal} below, we start with the continuous theory, i.e., the distance $\dst$ on $\Prob_2(\R)$ from \eqref{eq:dstformal} --- dual formulation, geodesics, relation to Martingale optimal transport etc --- and the corresponding gradient flows \eqref{eq:gfformal}. %
In the rigorous part, Sections \ref{sct:discrete} and \ref{sct:lambdacontract}, we first define the distance $\dst$ on the discretized space $\Prob_2(\delta\Z)$, we derive a sufficient criterion for measures being at finite distance, and we prove the existence of geodesics. Then we prove Theorem~\ref{thm:contractivity} about the mesh-uniform contractivity of the discretized equations.

\section{The diffusive transport metric: a formal overview}
\label{sct:formal}

\subsection{Definition and first properties}
The dynamical formulation \eqref{eq:dstformal} can be recast as a convex minimization problem with a linear constraint on the space of measures.
\begin{definition}
    The space of centered probability measures of finite second moment is  
    \[ \Prob_{2,c}(\R)=\set{\mu \in \Prob_2(\R): \int x \dx \mu(x)=0}. \]
    For $\mu_0,\mu_1\in\Prob_{2,c}(\R)$, define their \emph{diffusive transport distance} $\dst(\rho_0,\rho_1)\in[0,\infty]$ by
    \begin{align}
        \label{eq:newBB}
        \dst(\mu_0,\mu_1)^2 
        = \inf \left\{ \int_0^1\intr \left|\frac{\dn\omega_s}{\dn\mu_s}\right|^2\dd\mu_s\dd s\ \middle|\ \partial_s\mu_s=\pxx\omega_s \right\},
    \end{align}
    where the infimum runs over all weakly continuous curves  $(\mu_s)_{s\in[0,1]}$ of probability measures $\mu_s\in\Prob_{2,c}(\R)$ connecting $\mu_0$ to $\mu_1$, and all measurable curves $(\omega_s)_{s\in[0,1]}$ of Radon measures $\omega_s$ that satisfy the second order continuity equation $\partial_s\mu_s=\pxx\omega_s$ in the distributional sense. The convention is that the spatial integral is $+\infty$ for those $s\in[0,1]$ at which $\omega_s$ is not absolutely continuous with respect to $\rho_s$.
\end{definition}
Restricting $\dst$'s definition to measures $\mu_0,\mu_1\in\Prob_{2,c}(\R)$, i.e., of vanishing first moment and finite second moment, is natural since the infimum in \eqref{eq:newBB} is expected to be infinity if only one of $\mu_0$ and $\mu_1$ has finite second moment, or if their centers of mass differ. We conjecture that $\dst$ is a metric on $\Prob_{2,c}(\R)$. Currently, we cannot prove that $\dst$ is finite on all of $\Prob_{2,c}(\R)$; some sufficient criterion for finiteness is given later in Section \ref{sct:MartingaleOT}. In any case, conditionally on finiteness, $\dst$ satisfies the axioms of a metric:
\begin{proposition}
  For $\mu_0,\mu_1\in\Prob_2(\R)$ with $\dst(\mu_0,\mu_1)<\infty$ the infimum in \eqref{eq:newBB} is attained. Moreover, $\dst$ is a (pseudo)-metric $\Prob_2(\R)$ and satisfies for any $\mu_0,\mu_1,\mu_2\in\Prob_{2,c}(\R)$ with $\dst(\mu_1,\mu_0)<\infty$ and $\dst(\mu_1,\mu_2)<\infty$:
  \begin{enumerate}[ (i) ]
  \item \emph{positivity:}
  $\dst(\mu_0,\mu_1)\geq 0$ and $\dst(\mu_0,\mu_1)= 0$ only if $\mu_0=\mu_1$;
  \item \emph{symmetry:}
  $\dst(\mu_1,\mu_0) = \dst(\mu_0,\mu_1)$;
  \item \emph{triangle inequality:}  
    $\dst(\mu_0,\mu_2) \leq \dst(\mu_0,\mu_1)+\dst(\mu_1,\mu_2)$.
  \end{enumerate}
\end{proposition}
\begin{proof}
  Below we summarize the key ingredients for the proof. To turn this sketch into a rigorous proof, one can proceed in analogy to \cite{DNS}.
  
  The integral functional in \eqref{eq:newBB} is convex, since the quotient $w^2/r$ --- read as zero for $w=0$, and as $+\infty$ for $\rho=0$ and $w\neq0$ --- is jointly convex in the pair $(w,r)\in\R\times\Rnn$. The PDE constraint is linear. Thus, see e.g. \cite{AmbrosioButtazzo88}, the minimization problem \eqref{eq:newBB} is sequentially lower semi-continuous with respect to weak convergence of $(\mu,\omega)$ as measures on $[0,1]\times\R$. If $(\mu_s,\omega_s)_{s\in[0,1]}$ is an admissible curve of finite integral value, then the second moment of $\mu_s$ is $s$-uniformly controlled, due to the following estimate for arbitrary $0\le s_1\le s_2\le 1$:
  \begin{align*}
    &\frac12 \left|\intr|x|^2\dd\mu_{s_1} - \intr|x|^2\dd\mu_{s_0}\right|
    =  \frac12\left|\int_{s_0}^{s_1}\intr |x|^2\pxx\omega_s\dd x\dd s\right|
    = \left|\int_{s_0}^{s_1}\omega_s(\R)\dd s\right| \\
    & \qquad \le \left(\int_{s_0}^{s_1}\intr\left|\frac{\dn\omega_s}{\dn\mu_s}\right|^2\dd\mu_s\dd s\right)^{1/2}\left(\int_{s_0}^{s_1}\mu_s(\R)\dd s\right)^{1/2}
    \le \int_0^1\intr \left|\frac{\dn\omega_s}{\dn\mu_s}\right|^2\dd\mu_s\dd s.
  \end{align*}
  This proves weak sequential compactness of $\mu^k$ --- and a posteriori also of $\omega^k$ --- as measure on $[0,1]\times\R$, for any minimizing sequence $(\mu^k,\omega^k)$ in \eqref{eq:newBB}.
  In conclusion, if there is at least one admissible $(\mu_s,\omega_s)_{s\in[0,1]}$ with finite action, the direct methods from the calculus of variations yield the existence of a minimizer.

\smallskip
\noindent
  (i) \emph{Positivity:} The non-negativity is immediate from the definition. If $(\mu_s,\omega_s)_{s\in[0,1]}$ attains the infimum for $0=\dst(\rho_0,\rho_1)$ in \eqref{eq:newBB}, then it follows that $\mu$-a.e., $\dn\omega/\dn\mu$ vanishes. This means that $\omega_s$ is the zero measure, and so $\mu_s$ is independent of $s$, in particular $\mu_0=\mu_1$. 
  
  \smallskip
  \noindent
  (ii) \emph{Symmetry:} if $(\mu_s,\omega_s)_{s\in[0,1]}$ is an admissible pair connecting $\mu_0$ to $\mu_1$, then $(\mu_{1-s},-\omega_{1-s})_{s\in[0,1]}$ is admissible connecting $\mu_1$ to $\mu_0$, and vice versa, and the integral values are identical.
  
  \smallskip
  \noindent
  (iii) \emph{Triangle inequality:} let $(\hat\mu_s,\hat\omega_s)_{s\in[0,1]}$ realize the infimum for $\dst(\rho_0,\rho_1)$ in \eqref{eq:newBB}, and let $(\check\mu_s,\check\omega_s)_{s\in[0,1]}$ realize the infimum for $\dst(\rho_1,\rho_2)$. Define the parameter
  \begin{align*}
    \sigma := \frac{\dst(\mu_0,\mu_1)}{\dst(\mu_0,\mu_1)+\dst(\mu_1,\mu_2)},
  \end{align*}
  and then define the pair $(\tilde\mu_s,\tilde\omega_s)_{s\in[0,1]}$ by
  \begin{align*}
    \tilde\mu_s =
    \begin{cases}
      \hat\mu_{\frac s\sigma} & \text{for $0\le s\le\sigma$}, \\
      \check\mu_{\frac{s-\sigma}{1-\sigma}} & \text{for $\sigma\le s\le1$},
    \end{cases}
                                          \quad
                                          \tilde\omega_s =
                                          \begin{cases}
                                            \sigma^{-1}\hat\omega_{\frac s\sigma} & \text{for $0\le s<\sigma$}, \\
                                            (1-\sigma)^{-1}\check\omega_{\frac{s-\sigma}{1-\sigma}} & \text{for $\sigma\le s\le1$}.
                                          \end{cases}
  \end{align*}
  By means of the chain rule, it is easily verified that this pair $(\tilde\mu_s,\tilde\omega_s)_{s\in[0,1]}$ is admissible in the minimization problem \eqref{eq:newBB} for $\dst(\rho_0,\rho_2)$, and that
  \begin{equation*}
    \int_0^1\intr \left|\frac{\dn\tilde\omega_s}{\dn\tilde\mu_s}\right|^2\dd\tilde\mu_s\dd s
    = \left(\dst(\mu_0,\mu_1)+\dst(\mu_1,\mu_2)\right)^2. \qedhere
  \end{equation*}
\end{proof}

\subsection{Geodesics and dual formulation}
A geodesic between $\mu_0,\mu_1\in\Prob_{2,c}(\R)$ is a curve $(\mu_s)_{s\in[0,1]}$ such that $\dst(\mu_s,\mu_0)=s\,\dst(\mu_1,\mu_0)$ and $\dst(\mu_1,\mu_s)=(1-s)\,\dst(\mu_1,\mu_0)$ for all $s\in[0,1]$. Equivalently, $(\mu_s)_{s\in[0,1]}$ is a geodesic if there is an accompanying $(\omega_s)_{s\in[0,1]}$ such that $(\mu_s,\omega_s)_{s\in[0,1]}$ is a minimizer in \eqref{eq:newBB}. For the derivation of geodesic equations, we consider the dual problem for \eqref{eq:newBB}.
\begin{proposition}\label{prop:dual}
  For $\mu_0,\mu_1\in\Prob_2(\R)$, it holds
  \begin{align}
    \label{eq:newBBdual}
    \dst(\mu_0,\mu_1)^2 
    = \sup \left\{
    \intr \varphi_1\dd\mu_1 - \intr\varphi_0\dd\mu_0\ \middle|\ 
    \partial_s\varphi_s +  \frac12\big(\pxx\varphi_s\big)^2 \leq 0 \right\},
  \end{align}
  where the supremum runs over all bounded functions $\varphi\in C^2([0,1]\times\R)$.
\end{proposition}
\begin{proof}[Sketch of proof]
  We rewrite the definition \eqref{eq:newBB} as a saddle point problem,
  \begin{align*}
    \frac12\dst(\mu_0,\mu_1)^2
    = \inf_{(\mu,\omega)}\sup_{\varphi}\left[\frac12\int_0^1\intr \left|\frac{\dn\omega_s}{\dn\mu_s}\right|^2\dd\mu_s\dd s 
    + \int_0^1\intr \varphi_s\big(\partial_s\mu_s+\pxx\omega_s\big)\dd x\dd s\right].
  \end{align*}
  Rewriting the PDE using integration by parts yields
  \begin{equation}
    \label{eq:formal.saddle}
    \begin{split}
    \frac12\dst(\mu_0,\mu_1)^2
    = \inf_{(\mu,\omega)}\sup_{\varphi}&\ \bigg[\intr \varphi_1\dd\mu_1 - \intr\varphi_0\dd\mu_0 \\
      &+ \int_0^1\intr\left(\frac12\left|\frac{\dn\omega_s}{\dn\mu_s}\right|^2-\partial_s\varphi_s\right)\dd\mu_s\dd s + \int_0^1\intr\omega_s\pxx\varphi_s\dd x\dd s\bigg].      
    \end{split}
  \end{equation}
  The integrand is convex in $(\mu_s,\omega_s)_{s\in[0,1]}$, and is concave in $\varphi$, hence von Neumann's minimax theorem allows interchanging infimum and supremum. The necessary conditions for a minimum with respect to $\omega_s$ and with respect to $\mu_s$ imply
  \begin{align*}
    0=\omega_s + (\pxx\varphi_s)\,\mu_s,
    \quad
    0\le -\frac12\left|\frac{\dn\omega_s}{\dn\mu_s}\right|^2-\partial_s\varphi_s,
  \end{align*}
  with equality on the support of $\mu_s$. The inequality results from the constraint $\mu_s\ge0$. In combination, these condition imply that
  \begin{align*}
    \partial_s\varphi_s + \frac12(\pxx\varphi_s)^2 \leq 0 ,
  \end{align*}
  with equality on $\mu_s$'s support. And a substitution into~\eqref{eq:formal.saddle} eliminates all $s$-integrals, reducing it to~\eqref{eq:newBBdual}.
\end{proof}
\begin{remark}\label{rem:viscosity}
	The duality result of Proposition~\ref{prop:dual} is analogous to~\cite[Theorem 4.3]{HuesmannTrevisan2019} for the martingale optimal transport problem. For a justification of sufficiency of the one-sided inequality for $\varphi$ without equality constraint, see~\cite[Remark 4.5]{HuesmannTrevisan2019}.
\end{remark}
Assuming strong duality of \eqref{eq:newBB} with \eqref{eq:newBBdual}, any geodesic $(\mu_s)_{s\in[0,1]}$ possesses a corresponding maximizer $(\varphi_s)_{s\in[0,1]}$, satisfying the inequality constraint with equality $\mu$-a.e. This implies the following system of geodesic equations:
\begin{align}
    \label{eq:geodesics}
    \partial_s\mu_s = \pxx(\mu_s\pxx\varphi_s),
    \quad
    \partial_s\varphi_s = -\frac12\big(\pxx\varphi_s\big)^2.
\end{align}
\begin{example}
    There is a class of explicit geodesics, generated by the heat semigroup. If $(\rho_s)_{s\in[0,1]}$ is a distributional solution to the (possibly time-reversed) linear heat equation $\partial_s\rho=\alpha\pxx\rho$ with some $\alpha\in\R$, then $(\mu_s)_{s\in[0,1]}$ is a $\dst$-geodesic connecting $\mu_0$ to $\mu_1$, and
    \begin{align*}
        \dst(\rho_0,\rho_1)^2 = \int_0^1 \intr \frc{(\alpha\rho_s)}{\rho_s} \dd x\dd t =  \alpha^2.
    \end{align*}
    A corresponding $(\varphi_s)_{s\in[0,1]}$ satisfying the geodesic system \eqref{eq:geodesics} is given by $\varphi_s(x)=\frac\alpha2|x|^2-\frac{\alpha^2}2t$; note that the time-dependent constant term is irrelevant for the geodesic.

    There is an intuitive analogy to optimal mass transport: the heat semigroup plays the same role for the new distance $\dst$ as the translation group plays for the $L^2$-Wasserstein distance $\wass$.
\end{example}
The geodesic equations \eqref{eq:geodesics} motivate --- in analogy to the celebrated ``Otto calculus'' --- the introduction of a formal Riemannian structure: let $(\rho_t)_{|t|<\eps}$ be a curve with derivative $\zeta:=\partial_t\rho_0$ at $t=0$. Then the corresponding tangent vector is associated with the solution $\varphi$ to the following fourth order elliptic equation,
\begin{align*}
    \pxx(\rho_0\,\pxx\varphi) = \zeta,
\end{align*}
and the norm of that vector in the corresponding tangent space amounts to
\begin{align}
    \label{eq:Riemann}
    \|\zeta\|_{\rho_0}^2 = \intr (\pxx\varphi)^2\dd\rho_0.
\end{align}
Polarization of this norm yields formally a scalar product on the tangent bundle to $\Prob_{2,c}(\R)$. The corresponding gradient of a functional $\fnc$ is then given by minus the right-hand side of the PDE \eqref{eq:gfformal} from the introduction, justifying a posteriori to consider it as the gradient flow equation for $\fnc$ with respect to $\dst$.

\subsection{Comparison with the Martingale optimal transport metric}
\label{sct:MartingaleOT}
For two probability measures $\mu_0,\mu_1\in\Prob_2(\R)$, their \emph{Martingale optimal transport distance} $T_{\Mart}(\mu_0,\mu_1)$ is defined via the following dynamic formulation, see e.g.~\cite{HuesmannTrevisan2019},
\begin{align*}
  T_{\Mart}(\mu_0,\mu_1)^2
  = \inf_{(\mu,\omega)} \left\{ \int_0^1 \intr \left|\frac{\dn\omega_s}{\dn\mu_s}\right|^2 \dd\mu_s\dd s \,\middle|\,  \partial_s \mu_s = \pxx\omega_s,\ \omega_s\geq 0\right\} .
\end{align*}
The difference to definition \eqref{eq:newBB} of $\dst$ is the additional constraint $\omega_s\ge0$. Introducing $a:=\dn\omega_s/\dn\mu_s$, this means that $\mu_s$ satisfies the (possibly degenerate) \emph{parabolic} equation $\partial_s\mu_s=\pxx(a_s\mu_s)$. A trivial consequence is:
\begin{lemma}\label{lem:D:Mart}
  If $\mu_0,\mu_1\in\Prob_2(\R)$ have finite Martingale optimal transport distance $T_{\Mart}(\mu_0,\mu_1)<\infty$, then also their diffusive transport distance $\dst(\mu_0,\mu_1)$ is finite, and
  \begin{align*}
    \dst(\mu_0,\mu_1) \le T_{\Mart}(\mu_0,\mu_1).
  \end{align*}
\end{lemma}
Recall that two measures $\mu_0,\mu_1\in \Prob(\R)$ are in convex order, $\mu_0\leqcx \mu_1$, if $\intr \varphi \dx\mu_0 \leq \intr \varphi \dx\mu_1$ for all convex functions $\varphi: \R\to \R$.
\begin{corollary}
    If $\mu_0,\mu_1\in\Prob_2(\R)$ are in convex order and $\int \abs{x}^q \dx\mu_1<\infty$ for some $q>4$, then $\dst(\mu_0,\mu_1)<\infty$.
\end{corollary}
\begin{proof}
  This follows from the respective result in \cite[Proposition 5.1]{HuesmannTrevisan2019} and the comparison from Lemma~\ref{lem:D:Mart}.
\end{proof}

We close this section by arguing that there are gradient flows with respect to $\dst$, which do not preserve the convex order in time.
\begin{example}\label{ex:DLSS:noconvexorder}
    Consider the (locally smooth and positive) solution $\rho_t$ to the DLSS equation \eqref{eq:DLSS} with initial datum $\rho_0(x)=Z^{-1}\exp(x^2-x^4)$, with a proper normalization factor $Z$. Further, choose $\varphi(x)=\sqrt{\mu^2+x^2}$ as convex test function. Then
    \begin{align*}
        \frac{\dd}{\dn t}\intr \varphi\rho_t\dd x 
        = -\intr \pxx\varphi\,\pxx\log\rho_t\rho_t\dd x
        = \intr\pxx\varphi\frac{\partial_x\rho_t^2}{\rho_t}\dd x 
        - \intr \pxx\varphi\pxx\rho_t\dd x.
    \end{align*}
    This holds in particular at $t=0+$. Observe that $\pxx\varphi(x)=\mu^2(\mu^2+x^2)^{-3/2}$ approximates a Dirac measure at $x=0$ as $\mu\downarrow0$. Since $\rho_0(0)=Z^{-1}>0$ but $\partial_x\rho_0(0)=0$, the first integral above can be made arbitrarily small for $\mu\downarrow0$. On the other hand, the second integral tends to $-\pxx\rho_0(0)=-2Z^{-1}<0$. That is, $\rho_0$ and $\rho_t$ are not in positive convex order for small $t>0$.
\end{example}

\section{Discretization}
\label{sct:discrete}

\subsection{Difference operators}
\label{sct:diffop}
We consider an equidistant discretization of $\R$ with intervals of length $\delta>0$. The intervals are labeled by $\kappa\in\Z$, with $I_\kappa=((\kappa-\frac12)\delta,(\kappa+\frac12)\delta)$ being the $\kappa$th interval. Functions $f:\Z\to\R$ are interpreted as piecewise constant, attaining the constant value $f_\kappa:=f(\kappa)$ on $I_\kappa$, the integral of $f:\Z\to\R$ is thus
\begin{align*}
	\delta\sum_\kappa f := \delta\sum_{\kappa\in\Z} f_\kappa .
\end{align*}
We write $f_+,f_-:\Z\to\R$ for the left/right translates of $f:\Z\to\R$, i.e., $\big[f_\pm\big]_\kappa = f_{\kappa\pm1}$ for all $\kappa\in\Z$. The forward/backward difference quotient operators and the discrete Laplacian are defined in the usual way,
\begin{align*}
	\partial^\delta_\pm f := \frac{f_\pm-f}\delta, \qquad
	\ddff f := \frac{f_++f_--2f}{\delta^2}.
\end{align*}
Note that $\ddff = \partial^\delta_+\partial^\delta_-= \partial^\delta_-\partial^\delta_+$. For later reference, we recall the product rule %
\begin{align}
	\label{eq:product}
	\ddff(fg) = f\,\ddff g + g\,\ddff f + \partial^\delta_+f\,\partial^\delta_+g + \partial^\delta_-f\,\partial^\delta_-g,
\end{align}
and the local integration by parts rule
\begin{equation}\label{eq:ibp-gen}
    \begin{split}
    	&\delta\sum_{|\kappa|\le K} f\ddff g = \delta\sum_{|\kappa|\le K} g\ddff f + \bdry[K]{f}{g}, 
	   \quad \text{with the boundary term} \\
	   &\bdry[K]{f}{g} := \frac1{\delta}\left(f_{-K}g_{-(K+1)}-f_{-(K+1)}g_{-K}-f_{K+1}g_K+f_Kg_{K+1}\right). 
    \end{split}
\end{equation}
For later reference, we further introduce the convolution of two grid functions $f,g:\delta\Z\to\R$ by
\begin{align}
  \label{eq:dconvolve}
  [f\ast_\delta g]_\kappa := \delta\sum_\lambda f_\lambda g_{\kappa-\lambda}.
\end{align}

\subsection{Spaces of discrete functions}
Let $\ell^p$ for $p\ge1$ be the space of doubly infinite absolutely $p$-summable sequences $f:\Z\to\R$, with norm
\begin{align*}
	\|f\|_{\ell^p}^p = \delta\sum_\kappa |f_\kappa|^p.
\end{align*} 
Define $f$'s first and second moment, respectively, by
\begin{align*}
	\dcent{f} =
	\delta\sum_{\kappa\in\Z}(\delta\kappa)f_\kappa,
	\quad
	\dmom{f} = \delta\sum_{\kappa\in\Z}(\delta\kappa)^2f_\kappa,
\end{align*}
provided these series converge absolutely. The fundamental space for the following is 
\begin{align*}
	\delltwo := \left\{f\in\ell^1(\Z)\,\middle|\,\dmom{|f|}<\infty,\ \dcent{f}=0\right\},
\end{align*}
which contains those $f\in\ell^1$ of finite second and vanishing first moment. The space~$\delltwo$ is a Banach space with respect to the norm
\begin{align*}
	\|f\|_{\ell^1_2} = \|f\|_{\ell^1} + \dmom{|f|} = \delta\sum_\kappa [1+(\delta\kappa)^2]|f_\kappa|.
\end{align*}
We introduce further the corresponding subspace of probability densities
\begin{align*}
	\dprbtwo:=\biggl\{ \rho\in\delltwo\,\bigg|\,
	\rho_\kappa\ge0\ \text{for all $\kappa\in\Z$},\  \delta\sum_{\kappa\in\Z}\rho_\kappa = 1 , \  \dcent{\rho}=0 \biggr\}.
\end{align*}
Note that the tangent space to $\dprbtwo$ at any point $\rho$ with $\rho_\kappa>0$ for all $\kappa\in\Z$ is $\delltwo$.
\begin{lemma}
	Given $f\in\delltwo$ and a weight function $\omega:\Z\to\R$ with $|\omega_\kappa|\le C[1+(\delta\kappa)^2]$, then the following integration by parts rule holds:
	\begin{align*}
		\delta\sum_{\kappa\in\Z}f\ddff\omega = \delta\sum_{\kappa\in\Z}\omega\ddff f.
	\end{align*}
\end{lemma}
\begin{proof}
	In view of \eqref{eq:ibp-gen}, it suffices to verify that $\lim_{K\to\infty}\bdry[K] f \omega =0$.
    This follows since $|\omega_{\kappa\pm1}|\le4C[1+(\kappa\delta)^2]$, and since $[1+(\delta\kappa)^2]|f_\kappa|$ is summable by hypothesis.
\end{proof}

\subsection{Discrete heat flow}
Below, we need a particular semigroup on $\dprbtwo$ that is a discrete version of the heat flow. For each $\tau\ge0$, define $G^{\delta,\tau}\in\dprbtwo$ by 
\begin{align}
	\label{eq:dheatk}
	G^{\delta,\tau}_\kappa := \frac1{2\pi\delta}\int_{-\pi}^\pi \cos(\kappa z)e^{-\frac{2\tau}{\delta^2}(1-\cos z)}\dd z,
\end{align}
which is also expressible in terms of Bessel functions. Then $G^{\delta,0}_0=\delta^{-1}$, and $G^{\delta,0}_\kappa=0$ for $\kappa\neq0$ at $\tau=0$, and for all $\tau>0$, the discrete linear diffusion equation $\dn G^{\delta,\tau}/\dn\tau = \ddff G^{\delta,\tau}$ is satisfied.
Given $f\in\delltwo$ define for each $\tau\ge0$ the time-$\tau$-map of $f$ by $f^\tau=G^{\delta,\tau}\ast_{\delta}f\in\delltwo$ by means of discrete convolution \eqref{eq:dconvolve}.
Then $f^\tau$ is the unique solution to the spatially discrete heat equation with initial value $f$, 
\begin{align}
    \label{eq:dheatflow}
    \frac{\dd}{\dn\tau} f^\tau = \ddff f^\tau, \quad\text{and}\quad f^0=f.     
\end{align}
We use the proximity of $G^\delta$ to the (continuous) heat kernel $G^{0,\tau}(\xi)= (4\pi\tau)^{-1/2}\exp(-\xi^2/(4\tau))$.
Indeed, the difference can be estimated as follows, with a universal constant $C_\text{heat}>1$,
\begin{align}
	\label{eq:YIP}
	G^{0,\tau}(\delta\kappa) \le G^{\delta,\tau}_\kappa \le C_\text{heat} G^{0,\tau}(\delta\kappa) ,
\end{align}
see e.g.~\cite[Proposition 3]{MisiatsYip}.

\subsection{Inverse Laplacian}
For $f:\Z\to\R$, define $(\ddff)^{-1}f:\Z\to\R$ as
\begin{align}
	\label{eq:definverse}
	(\ddff)^{-1}_\kappa f = \delta\sum_{\lambda>\kappa}[\delta(\lambda-\kappa)]f_\lambda,
\end{align}
provided that series converges. The thus defined operator $(\ddff)^{-1}$ is indeed inverse to $\ddff$ in the following specific sense.
\begin{lemma}
	\label{lem:dLaplace}
	For each $f\in\ell^1(\Z)$ with $\delta\sum_\kappa |\delta\kappa||f_\kappa|<\infty$ and $\delta\sum_\kappa f_\kappa=\delta\sum_\kappa(\delta\kappa)f_\kappa=0$, the function $(\ddff)^{-1}f$ is well-defined, and it is the unique solution $u:\Z\to\R$ to
	\begin{align}
		\label{eq:dLaplace1}
		\ddff u = f
	\end{align}
	subject to the condition that
	\begin{align}
		\label{eq:dLaplace2}
		\lim_{\kappa\to\pm\infty}u_\kappa = 0.
	\end{align}
\end{lemma}
\begin{proof}
	Convergence of the series in \eqref{eq:definverse} at each $\kappa$ is guaranteed by $\delta\sum_\kappa|\delta\kappa||f_\kappa|<\infty$. Further note that, since $\delta\sum_\kappa f_\kappa=\delta\sum_\kappa(\delta\kappa)f_\kappa=0$,
	\begin{align}
		\label{eq:leftright}
		(\ddff)^{-1}_\kappa f = -\delta\sum_{\lambda<\kappa}[\delta(\lambda-\kappa)]f_\lambda.
	\end{align}    
	From \eqref{eq:definverse}, it follows immediately that $u:=(\ddff)^{-1}f$ is indeed as solution to \eqref{eq:dLaplace1}:
	\begin{align*}
		\ddff_\kappa u 
		&= \frac1{\delta^2}\big((\ddff)^{-1}_{\kappa+1}f+(\ddff)^{-1}_{\kappa-1}f-2(\ddff)^{-1}_\kappa f\big) \\
		&= \sum_{\lambda>\kappa+1}(\lambda-\kappa-1)f_\lambda + \sum_{\lambda>\kappa-1}(\lambda-\kappa+1)f_\lambda - 2\sum_{\lambda>\kappa}(\lambda-\kappa)f_\lambda \\
		&= f_{\kappa+1}-2f_{\kappa+1}  + f_\kappa + f_{\kappa+1} = f_\kappa.
	\end{align*}
	Next, we verify \eqref{eq:dLaplace2}: for $\kappa\to+\infty$, using again that $\delta\sum_\kappa|\delta\kappa||f_\kappa|<\infty$,
	\begin{align*}
		\lim_{\kappa\to+\infty}|u_\kappa| 
		\le \lim_{\kappa\to\infty}\delta\sum_{\lambda>\kappa}|\delta(\lambda-\kappa)||f_\lambda|
		\le 2\lim_{\kappa\to\infty}\delta\sum_{\lambda>\kappa}|\delta\lambda||f_\lambda|
		=0.
	\end{align*}
	And for $\kappa\to-\infty$, we use the analogous estimate on the alternative representation \eqref{eq:leftright}.
	
	Finally, to see that $u=(\ddff)^{-1}f$ is the only solution to \eqref{eq:dLaplace1}\&\eqref{eq:dLaplace2}, it suffices to observe that any two solutions to \eqref{eq:dLaplace1} differ by a function $g$ of the form $g_\kappa=a+b\delta\kappa$ with parameters $a,b\in\R$. The only choice compatible with \eqref{eq:dLaplace2} is the one above.
\end{proof}

\subsection{Discrete optimal diffusive connection}
We essentially follow the ideas developed in \cite{DNS} for definition of modified Wasserstein distances with non-linear mobilities, see also \cite{CLSS} for an application in the context of metric gradient flows. However, in the spatially discrete context at hand, many of the technical details simplify.

Introduce the convex and lower semi-continuous function $\frc{(\cdot)}{(\cdot)}:\Rnn\times\R\to\Rnn\cup\{+\infty\}$ to the extended reals by
\begin{align*}
	\frc{p}{r} = 
	\begin{cases}
		\frac{p^2}r & \text{if $r>0$}, \\
		0 & \text{if $r=0$ and also $p=0$}, \\
		+\infty & \text{if $r=0$ but $p\neq0$}.
	\end{cases}
\end{align*}
Introduce further the space 
\[ \curves = \left\{(\rho,\welo):[0,1]\to\delltwo\times\ell^2\,\middle|\,(\rho_\kappa,\welo_\kappa)\in H^1([0,1])\times L^2([0,1])\,\text{f.a. $\kappa\in\Z$},\,\dot\rho=\ddff\welo\right\},
\]
where the discrete diffusion equation is imposed in the sense that
\begin{align}
	\label{eq:dcont}
	\dot\rho_\kappa(s) = \ddff_\kappa\welo(s)\quad \text{for all $\kappa\in\Z$ and at almost every $s\in[0,1]$,}
\end{align}
as well as the subspace of fixed end points $\hat\rho,\check\rho\in\dprbtwo$:
\[ 
    \curves(\hat\rho,\check\rho) = \left\{ (\rho,\welo)\in\curves\,\middle|\,\rho(0)=\hat\rho,\,\rho(1)=\check\rho\right\}.
\]
\begin{definition}[Weak convergence in $\curves$]
	\label{dfn:weakC}
	We say that a sequence $(\rho^\eps,\welo^\eps)\in\curves$ converges weakly in~$\curves$ towards a limit $(\rho^*,\welo^*)\in\curves$ if $\rho^\eps_\kappa\rightharpoonup\rho^*_\kappa$ weakly in $H^1([0,1])$ and $\welo^\eps_\kappa\rightharpoonup\welo^*_\kappa$ weakly in $L^2([0,1])$ for each $\kappa\in\Z$.    
\end{definition}
We emphasize that $\rho^*(t)\in\delltwo$ for a.e.~$t\in [0,1]$, and in particular the conservation of total mass, is part of the definition. Inheritance of the continuity equation \eqref{eq:dcont} by the limit is automatic. Further note that, thanks to the compact embedding $H^1([0,1])\hookrightarrow C([0,1])$, the curves $\rho_\kappa$ are actually continuous functions, and weak convergence of $(\rho^\eps,\welo^\eps)$ to $(\rho^*,\welo^*)$ in $\curves$ implies uniform convergence of $\rho^\eps_\kappa$ to $\rho^*_\kappa$, for each $\kappa\in\Z$.

For a pair $(\rho,\welo)\in\curves$, define its action by
\begin{align}
	\label{eq:defact}
	\daction(\rho,\welo) := \delta\sum_{\kappa\in\Z}\int_0^1\frc{\welo_\kappa(s)}{\rho_\kappa(s)} \dd s.
\end{align}
\begin{remark}[On the choice of $\curves$]
	In the approach of \cite{DNS}, the pairs $(\rho,\welo)$ in the respective definition of the action functional are arbitrary measurable time-dependent Borel and Radon measures, respectively, that are connected by the distributional continuity equation. It then needs to be shown, for instance, that finite action implies that $\rho$ can be chosen weakly continuous in time.
	
	The restrictions imposed on our space $\curves$ --- with $\rho_\kappa\in H^1([0,1])$ and $\welo_\kappa\in L^2([0,1])$ for each $\kappa$ --- are not only technically convenient, but come naturally with our discrete setting: since any $\rho\in\delltwo$ trivially satisfies the estimate $\rho_\kappa\le\frac1\delta$, one has
	\begin{align*}
		\delta\sum_\kappa\int_0^1 \welo_\kappa^2 \le \frac{\daction(\rho,\welo)}{\delta^2}
	\end{align*}
	directly from the definition \eqref{eq:defact}. So finite action immediately implies $\welo_\kappa\in L^2([0,1])$ for each $\kappa$, and a posteriori also $\rho_\kappa\in H^1([0,1])$ by means of \eqref{eq:dcont}.
\end{remark}
\begin{lemma}[Lower semi-continuity of the action]\label{lem:lsc}
	The action functional is lower semi-continuous with respect to weak convergence in $\curves$.
\end{lemma}
\begin{proof}
	Assume that $(\rho^\eps,\welo^\eps)$ converges to $(\rho^*,\welo^*)$ weakly in $\curves$. Then, by definition, we have (more than) weak-$\ast$-convergence in $L^1([0,1])$ of $\rho^\eps_\kappa$ to $\rho^*_\kappa$ and of $\welo^\eps_\kappa$ to $\welo^*_\kappa$, respectively, for each $\kappa\in\Z$. By convexity of the map   $(r,p)\mapsto\frc{p}{r}$, we conclude that, see e.g. \cite{AmbrosioButtazzo88},
	\begin{align*}
		\int_0^1\frc{\welo^*_\kappa(s)}{\rho^*_\kappa(s)}\dd s \le \liminf_\eps \int_0^1 \frc{\welo^\eps_\kappa(s)}{\rho^\eps_\kappa(s)}\dd s .
	\end{align*}
	And since the limes inferior of a sum is an upper bound on the sum of the individual limits, this yields the desired result
	\begin{equation*}
		\daction(\rho^*,\welo^*)
		\le \delta\sum_\kappa \liminf_\eps \int_0^1\frc{\welo^\eps_\kappa(s)}{\rho^\eps_\kappa(s)}\dd s
		\le \liminf_\eps \daction(\rho^\eps,\welo^\eps) .\qedhere
	\end{equation*}
\end{proof}
\begin{lemma}[Second moment estimate]
	If $(\rho,\welo)\in\curves$ has finite action $A:=\daction{\rho}{\welo}$, then the second moment $\dmom{\rho(s)}$ is a H\"older continuous function of $s$, with
	\begin{align}
		\label{eq:momest}
		\big|\dmom{\rho(s_1)}-\dmom{\rho(s_2)}\big| \le 2\sqrt A |s_1-s_2|^{1/2} \quad \text{for all $s_1,s_2\in[0,1]$.}
	\end{align}
\end{lemma}
\begin{proof}
	Each component $\rho_\kappa$ belongs to $H^1([0,1])$ and thus has an absolutely continuous representative, which is differentiable at almost every $s\in[0,1]$, and the classical and weak derivatives agree. Let $\varphi:\Z\to\R$ be zero except at finitely many indices $\kappa$. Then clearly $\Phi:[0,1]\to\R$ with
	\begin{align*}
		\Phi(s) = \delta\sum_\kappa \varphi_\kappa \rho_\kappa(s)
	\end{align*}
	inherits absolute continuity and differentiability almost everywhere. At each point $s\in(0,1)$ of differentiability,
	\begin{align*}
		\dot\Phi(s) 
		= \delta\sum_\kappa \varphi_\kappa \dot\rho_\kappa(s)
		= \delta\sum_\kappa \varphi_\kappa \ddff_\kappa\welo(s)
		= \delta\sum_\kappa \ddff_\kappa\varphi \welo_\kappa(s),
	\end{align*}
	and thus for any $0\le s_1 \le s_2\le 1$, by the fundamental theorem of calculus,
	\begin{align*}
		\big|\Phi(s_2)-\Phi(s_1) \big|
		&\le \int_{s_1}^{s_2} \big|\dot\Phi(s)\big|\dd s \\
		&\le \delta\sum_\kappa \left[\big|\ddff_\kappa\varphi \big| \int_{s_1}^{s_2}\big|\welo_\kappa(s)\big| \dd s\right] \\
		&\le \Big(\max_\lambda\big|\ddff_\lambda\varphi\big|\Big) 
		\left(\delta\sum_\kappa\int_{s_1}^{s_2}\rho_\kappa(s)\dd s\right)^{1/2} \left(\delta\sum_\kappa\int_0^1\frc{\welo_\kappa(s)}{\rho_\kappa(s)}\right)^{1/2} \\
		&\le \Big(\max_\lambda\big|\ddff_\lambda\varphi\big|\Big)
		(s_2-s_1)^{1/2}
		\sqrt{\daction(\rho,\welo)}.
	\end{align*}
	That estimate on $\Phi(s_2)-\Phi(s_1)$ still holds when $\varphi$ with only finitely many non-zero components is replaced by $\varphi+c$ for some constant $c$, since this only changes the values of $\Phi$ by $c$. We may thus choose for $\varphi$ an increasing approximation $\varphi^n$ of the unbounded function $\bar\varphi:\Z\to\R$ with $\bar\varphi_\kappa = (\delta\kappa)^2$, like
	\begin{align*}
		\varphi^n_\kappa = \begin{cases}
			(\delta\kappa)^2 & \text{for $|\kappa|\le n$}, \\
			2(\delta n)^2 - (2\delta n - \delta|\kappa|)^2 & \text{for $n\le|\kappa|\le2n$}, \\
			2(\delta n)^2 &\text{for $|\kappa|\ge2n$}.
		\end{cases}
	\end{align*}
	With that approximation, we have $|\ddff\varphi^n_\kappa|\le2$ for all $n$ and $\kappa$. Choosing $s_1:=0$ and using that $\dmom{\rho(0)}$ is finite since $\rho(0)=\hat\rho\in\delltwo$, the monotone convergence theorem implies finiteness of $\dmom{\rho(s_2)}$ at any $s_2\in[0,1]$. Finally, varying both $s_1$ and $s_2$ yields \eqref{eq:momest}.
\end{proof}
\begin{lemma}[Existence of geodesics]\label{lem:geodesic}
	Assume that $\hat\rho,\check\rho\in\dprbtwo$ are such that $\curves(\hat\rho,\check\rho)$ contains a curve of finite action. Then there is some $(\rho^*,\welo^*)\in\curves(\hat\rho,\check\rho)$ that minimizes $\daction(\rho,\welo)$ among all $(\rho,\welo)\in\curves(\hat\rho,\check\rho)$.
\end{lemma}
\begin{proof}
	Since there is at least one curve $(\bar\rho,\bar\welo)\in\curves(\hat\rho,\check\rho)$ of finite action, and since the action functional is non-negative, $\daction{\cdot}{\cdot}$ has a finite infimum $I\in[0,\daction{\bar\rho}{\bar\welo}]$ on $\curves(\hat\rho,\check\rho)$. We need to prove that $I$ is attained by some $(\rho^*,\welo^*)\in\curves(\hat\rho,\check\rho)$.
	
	Let $(\rho^\eps,\welo^\eps)$ a minimizing sequence, with $\daction{\rho^\eps}{\welo^\eps}\le I+1$. Since $\rho^\eps(s)\in\delltwo$, we have the trivial upper bound $\rho^\eps_\kappa(s)\le\frac1\delta$. It then follows from
	\begin{align*}
		\int_0^1 \welo_\kappa(s)^2\dd s 
		\le \frac1\delta\int_0^1\frc{\welo_\kappa^\eps(s)}{\rho^\eps_\kappa(s)}\dd s 
		\le \frac{I+1}{\delta}
	\end{align*}
	that each $\welo^\eps_\kappa$ is $\eps$-uniformly bounded in $L^2([0,1])$. Using a diagonal argument, there are a (non-relabeled) subsequence and a limit curve $\welo^*$, such that $\welo^\eps_\kappa\rightharpoonup\welo^*_\kappa$ in $L^2([0,1])$ for each $\kappa\in\Z$. Now, from the continuity equation $\dot\rho^\eps_\kappa = \ddff_\kappa\welo^\eps$ and since the boundary values $\rho^\eps_\kappa(0)=\hat\rho_\kappa$ are fixed, it follows that each $\rho^\eps_\kappa$ converges weakly in $H^1([0,1])$ --- and thus also strongly in $C([0,1])$ --- to a limit $\rho^*_\kappa$. 
	
	We still need to verify that $\rho^*(s)\in\delltwo$. This is a consequence of the moment estimate \eqref{eq:momest}: we have
	\begin{align}
		\label{eq:momabove}
		\dmom{\rho^\eps(s)} \le \dmom{\hat\rho} + 2\sqrt{I+1},
	\end{align}
	uniformly in $\eps$ and $s\in[0,1]$. Since 
	\begin{align}
		\label{eq:ptw}
		\rho^\eps_\kappa(s)\to\rho^*_\kappa(s) \quad  \text{for each $\kappa$ and $s$ as $\eps\to0$}, 
	\end{align}
	we conclude by lower semi-continuity of the second moment that $\dmom{\rho^\eps(s)}$ satisfies the same estimate \eqref{eq:momabove}, and in particular $\dmom{\rho^*(s)}<\infty$ for all $s\in[0,1]$. To obtain conservation of mass, it suffices that \eqref{eq:momabove} implies tightness of the $\rho^\eps(s)$: for each $\mu>0$, there is a $K$ such that $\sum_{|\kappa|>K}\rho^\eps(s)<\mu$ for all $\eps$ and $s\in[0,1]$, and also $\sum_{|\kappa|>K}\rho^*(s)<\mu$. By \eqref{eq:ptw}, we have further that
	\begin{align*}
		\left|\sum_{|\kappa|\le K}\rho^\eps_\kappa - \sum_{|\kappa|\le K}\rho^*_\kappa \right| < \mu
	\end{align*}
	for all sufficiently small $\eps$. It follows that $|1-\sum_\kappa \rho^*_\kappa|<3\mu$ for those $\eps$, and therefore $\sum_\kappa\rho^*_\kappa=1$ since $\mu$ is arbitrary. 
\end{proof}

\subsection{Curves of finite action}
We prove that any densities in $\dprbtwo$ with Gaussian tails can be connected by a curve of finite action. Our estimate on the action is uniform in $\delta$, thus the result carries over to the continuum limit.
\begin{proposition}\label{prop:existence:curve}
	Assume that there is an $\alpha>0$ for which $\rho^0,\rho^1\in\dprbtwo$ satisfy
	\begin{align}\label{eq:ass:GaussMoments}
		R^0:=\delta\sum_\kappa e^{\alpha(\delta\kappa)^2}\rho^0_\kappa <\infty,
		\quad
		R^1:=\delta\sum_\kappa e^{\alpha(\delta\kappa)^2}\rho^1_\kappa <\infty.
	\end{align}
    Then there is a connecting curve $(\rho,\welo)\in\curves(\rho_0,\rho_1)$ of finite action $\daction(\rho,\welo)\le A(R_0,R_1)$, with a bound $A(R_0,R_1)$ independent of $\delta$.
\end{proposition}
\begin{proof}
	Below, a define a specific element $(\rho^s,\welo^s)_{s\in[0,1]}\in\curves(\rho^0,\rho^1)$ and then show that it is of finite action. The construction uses the discrete heat kernel $G^\delta$ from \eqref{eq:dheatk}:
	\begin{itemize}
		\item For $0\le s\le\frac14$, let $\rho^s:=G^{\delta,\frac3\alpha s}\ast_\delta\rho^0$ and $\welo^s:=\frac3\alpha\rho^s$.
		\item For $\frac34\le s\le1$, let $\rho^s:=G^{\delta,\frac3\alpha(1-s)}\ast_\delta\rho^1$ and $\welo^s:=-\frac3\alpha\rho^s$.
		\item For $\frac14\le s\le\frac34$, let $\rho^s:=(2s-\frac12)\rho^{\frac34}+(\frac32-2s)\rho^{\frac14}$ and $\welo^s=2(\ddff)^{-1}(\rho_{3/4}-\rho_{1/4})$.
	\end{itemize}
    The constraint $\dot\rho=\ddff\welo$ is immediately verified, recalling the property \eqref{eq:dheatflow} of the heat kernel, and Lemma \ref{lem:dLaplace} about the inverse Laplacian.    
    On $0<s<1/4$, the action is computed easily:
    \begin{align*}
        \int_0^{1/4}\delta\sum \frac{\big(\welo^s_\kappa\big)^2}{\rho^s_\kappa}
		= \int_0^{1/4}\delta\sum_\kappa(3/\alpha)^2\rho^s_\kappa\dd s  
		= \frac9{4\alpha^2},
    \end{align*}
    and the same value is obtained on $3/4<s<1$.
    
	The estimate for the action on $\frac14\le s\le\frac34$ is more difficult. First, we show that the bounds \eqref{eq:YIP} on the discrete heat kernel $G^\delta$ imply that
	\begin{align}
		\label{eq:connectbelow}
		re^{-\frac{4\alpha}9(\delta\kappa)^2}
		\le\rho^s_\kappa 
        \le Re^{-\frac\alpha4(\delta\kappa)^2},
	\end{align}
	with positive constants $R>r>0$ independent of $\kappa$ and of $s\in[\frac14,\frac34]$. Since $\rho^s$ is the linear interpolation of $\rho^{1/4}$ and $\rho^{3/4}$, it actually suffices to prove \eqref{eq:connectbelow} for $s=1/4$ and for $s=3/4$. To prove the lower bound in \eqref{eq:connectbelow} at $s=1/4$, we use the elementary inequality
	\begin{align*}
		\frac\alpha3\delta^2(\kappa-\lambda)^2
		\le \frac{4\alpha}9(\delta\kappa)^2 + \frac{4\alpha}3(\delta\lambda)^2,
	\end{align*}
	to obtain by means of the lower bound in \eqref{eq:YIP} that
	\begin{align*}
		\rho^{1/4}_\kappa 
		= \sum_{\lambda} G_{\kappa-\lambda}^{\frac3{4\alpha}}\rho^0_\lambda
		\ge \frac\delta{\sqrt{3\pi/\alpha}} \sum_\lambda e^{-\frac\alpha3\delta^2(\kappa-\lambda)^2}\rho^0_\lambda
		\ge \left(\frac\delta{\sqrt{3\pi/\alpha}} \sum_\lambda e^{-\frac{4\alpha}3(\delta\lambda)^2}\rho^0_\lambda\right) e^{-\frac{4\alpha}9(\delta\kappa)^2}.
	\end{align*}
    The last sum can be estimated just in terms of $R^0$, independently of $\delta$: by H\"older's inequality,
    \begin{align*}
        1 = \delta\sum_\lambda \rho_\lambda \le \left(\delta\sum_\lambda e^{\alpha(\delta\lambda)^2}\rho^0_\lambda\right)^{4/7}\left(\delta\sum_\lambda e^{-\frac43\alpha(\delta\lambda)^2}\rho^0_\lambda\right)^{3/7},
    \end{align*}
    and therefore
    \begin{align*}
        \delta\sum_\lambda e^{-\frac43\alpha(\delta\lambda)^2}\rho^0_\lambda \ge r:= (R_0)^{-4/3}.
    \end{align*}
    To prove the upper bound in \eqref{eq:connectbelow} at $s=1/4$, we use instead the elementary inequality 
	\begin{align*}
		\frac\alpha3\delta^2(\kappa-\lambda)^2 \ge \frac\alpha4(\delta\kappa)^2 - \alpha(\delta\lambda)^2,
	\end{align*}
    and obtain this times by means of the upper bound in \eqref{eq:YIP} that
    \begin{align*}
		\rho^{1/4}_\kappa
		= \sum_{\lambda} G_{\kappa-\lambda}^{\frac3{4\alpha}}\rho^0_\lambda
		\le \frac{C_\text{heat}\delta}{\sqrt{3\pi/\alpha}}\sum_\lambda e^{-\frac\alpha3\delta^2(\kappa-\lambda)^2}\rho^0_\lambda
		\le \left(\frac{C_\text{heat}\delta}{\sqrt{3\pi/\alpha}}\sum_\lambda e^{\alpha(\delta\lambda)^2}\rho^0_\lambda\right)e^{-\frac\alpha4(\delta\kappa)^2}.
	\end{align*}
	The arguments at $s=3/4$ are analogous, finishing the proof of \eqref{eq:connectbelow}, with the ($\delta$-independent) constants
	\begin{align*}
		r = \max(R^0,R^1)^{-4/3}, \quad R = C_\text{heat}\sqrt{\frac\alpha{3\pi}}\max(R^0,R^1).
	\end{align*}
    Next, we show that
    \begin{align}
		\label{eq:connectfield}
		\big|\welo^s_\kappa\big| \le 4B_\alpha R e^{-\frac\alpha4(\delta\kappa)^2},
	\end{align}
	with some constant $B_\alpha$ depending on $\alpha$ only. Recall that, by definition of the interpolation and by the upper bound in \eqref{eq:connectbelow},
	\begin{align*}
		\big|\ddff\welo^s\big| = 2\big|\rho^{3/4}_\kappa-\rho^{1/4}_\kappa\big|\le 4Re^{-\frac\alpha4(\delta\kappa)^2}.  
	\end{align*}
	We show \eqref{eq:connectfield} for $\kappa>0$, combining the inequality just above with the following estimate for $\lambda>\kappa$,
	\begin{align*}
		\frac\alpha4(\delta\lambda)^2 = \frac\alpha4\delta^2((\lambda-\kappa)+\kappa)^2 \ge \frac{\alpha}{4}(\delta\kappa)^2+\frac\alpha4\delta^2(\lambda-\kappa)^2,
	\end{align*}
    to obtain 
	\begin{align*}
        \big|\welo^s_\kappa\big|
        &\le 2\delta\sum_{\lambda>\kappa}\delta(\lambda-\kappa)\big|\rho^{3/4}_\lambda-\rho^{1/4}_\lambda\big| 
		\le 4R \delta\sum_{\lambda>\kappa}\delta(\lambda-\kappa)e^{-\frac\alpha4(\delta\lambda)^2}\\
		&\le 4R\left(\delta\sum_{\lambda>\kappa}\delta(\lambda-\kappa)e^{-\frac\alpha4\delta^2(\lambda-\kappa)^2} \right) e^{-\frac\alpha4(\delta\kappa)^2}. 
	\end{align*}
	The last sum above is finite, and possesses a $\delta$-independent upper bound $B_\alpha$, thanks to the elementary inequality $ue^{-\beta u}\le(\beta e)^{-1}$ for arbitrary $u>0$ and $\beta>0$. The argument for $\kappa<0$ is analogous, using that since each $\rho^s$ is centered, $\welo^s$ admits the alternative representation
	\begin{align*}
		\welo^s_\kappa = \delta\sum_{\lambda<\kappa}\delta(\kappa-\lambda)\big[\rho^{3/4}_\lambda-\rho^{1/4}_\lambda\big].
	\end{align*}
	To finish up, we combine \eqref{eq:connectbelow} and \eqref{eq:connectfield} to obtain
	\begin{equation*}
		\delta\sum_\kappa \frac{\big(\welo^s_\kappa\big)^2}{\rho^s_\kappa}
		\le \frac{16 B_\alpha^2 R^2}{r} \delta\sum_{\kappa} e^{-\frac\alpha9(\delta\kappa)^2}. \qedhere
	\end{equation*}
\end{proof}
\begin{remark}[Weaker assumptions on marginals]
	The assumption~\eqref{eq:ass:GaussMoments} in Proposition~\ref{prop:existence:curve} asks for Gaussians tails for the marginals $\rho^0$ and $\rho^1$. By using a stochastic approach based on explicit solutions to the Skorokhod embedding problem on $\mathbb{Z}$~\cite{CoxObloj2008,HeHuOblojZhou2019} and a construction of a connecting curve similar as in the continuous case~\cite[Proposition 5.1]{HuesmannTrevisan2019}, it appears likely that any two measures in $\dprbtwo$ can be coupled with finite action, without additional hypotheses.
	Here, we have chosen to stay on the analytical side and use an interpolation done similarly for the martingale transport problem in~\cite[Proposition 5.1]{HuesmannTrevisan2019}.
\end{remark}

\section{Semi-contractive flows in the discrete metric}
In this section, we prove Theorem \ref{thm:contractivity} about the uniform semi-contractivity of the spatially discrete gradient flows \eqref{eq:dFextern}--\eqref{eq:dFpme}.

\subsection{\texorpdfstring{$\lambda$}{λ}-contractivity}
\label{sct:lambdacontract}
Formally, the gradient flow equation for a functional $\fnc$ in the metric $\dst$ is given by 
\begin{align}
    \label{eq:formalgf}
    \dot\rho_\kappa = \ddff_\kappa\big(\rho\,\ddff\fnc'(\rho)\big),
    \quad
    \big[\partial\fnc(\rho)\big]_\kappa = \frac{\partial\fnc}{\partial\rho_\kappa}(\rho).
\end{align}
In the following, we consider $\fnc$ such that global solutions to \eqref{eq:formalgf} exist for a dense of initial conditions (like strictly positive and rapidly decaying at infinity), and that the solution map extends to a continuous semi-group $\sgrp$ on $\dprbtwo$. 
\begin{definition}
    A continuous semigroup $\sgrp$ is $\lambda$-contractive with respect to the function $\fnc$ if it satisfies the evolutionary variational inequality for all choices of $\bar\rho$ and $\bar\eta$:
    \begin{align}
        \label{eq:evi}
        \frac12\frac{\dd}{\dn t}\bigg|_{t=0+}\dst\big(\sgrp^t(\bar\rho),\bar\eta)^2 + \frac\lambda2\dst(\bar\rho,\bar\eta)^2 
        \le \fnc(\bar\eta) - \fnc(\bar\rho).
    \end{align}
\end{definition}
In the proofs below, we use an equivalent integral characterization \eqref{eq:eviii} of $\lambda$-contractivity~\cite{DaneriSavare}, which is more robust than the differential form~\eqref{eq:evi} above, and is in particular accessible by means of ``Eulerian calculus''~\cite{OttoWest}. We will further use approximations of geodesics by very regular almost minimizers. 
\begin{definition}\label{dfn:regular}
    A curve $(\rho,\welo)\in\curves$ is \emph{regular} if
    \begin{itemize}
        \item $s\mapsto\rho_\kappa(s)$ and $s\mapsto\welo_\kappa(s)$ are smooth functions on $[0,1]$ for each $\kappa\in\Z$,
        \item $\rho_\kappa(s)>0$ for all $s\in[0,1]$ and $\kappa\in\Z$,
        \item there exists a constant $C$ such that, for all $s\in[0,1]$ and $\kappa\in\Z$,
        \begin{align}
            \label{eq:rhocompare}
            \max\big(\rho_{\kappa-1}(s),\rho_{\kappa+1}(s)\big) \le C\rho_\kappa.
        \end{align}
    \end{itemize}
\end{definition}
The existence of suitable approximations for geodesics is guaranteed by the follwing result.
\begin{lemma}[Smooth and positive approximation]\label{lem:smooth}
	If $(\rho^*,\welo^*)\in\curves$ has finite action, then there is a sequence of approximations $(\rho^\eps,\welo^\eps)\in\curves$ with the following properties:
	\begin{itemize}
        \item for each $\eps>0$, the curve $(\rho^\eps,\welo^\eps)$ is regular in the sense of Definition \ref{dfn:regular} above,
        \item $(\rho^\eps,\welo^\eps)$ converges weakly to $(\rho^*,\welo^*)$ in $\curves$ as $\eps\searrow0$ in the sense of Definition \ref{dfn:weakC},
		\item $\daction(\rho^\eps,\welo^\eps)$ converges to $\daction(\rho^*,\welo^*)$ from below as $\eps\searrow0$.
	\end{itemize}
\end{lemma}
Before proving Lemma \ref{lem:smooth}, we formulate its main conclusion, which is an amendable sufficient differential relation.
\begin{lemma}[Sufficient condition for $\lambda$-contractive gradient flow]\label{lem:lambda-contractive}
    Consider parametrized pairs $(\rho(s,t),\welo(s,t))\in\Rp\times\R$ depending on $s\in(0,1)$ and $t\in(0,\tau)$ for some $\tau>0$ with the following properties:
    \begin{itemize}
        \item $(\rho(\cdot,t),\welo(\cdot,t))\in\curves$ is a regular curve in the sense of Definition \ref{dfn:regular} for each $t\in(0,\tau)$;
        \item $(\rho_\kappa(s,\cdot),\welo_\kappa(s,\cdot))$ is differentiable for each $s\in(0,1)$ and $\kappa\in\Z$;
        \item the following differential equations are satisfied:
        \begin{align}\label{eq:sideode}
	           \partial_t\rho(s,t) = s\ddff F(\rho(s,t)),\quad \partial_t\welo(s,t) = F(\rho(s,t))+s \mathrm{D}F(\rho(s,t))[\ddff\welo(s,t)],
        \end{align}
        where $F(\rho):=\rho\ddff\fnc'(\rho)$, and $\mathrm{D}F(\rho)[\xi]$ is $F$'s Fréchet derivative at $\rho$ in the direction $\xi$. 
    \end{itemize}
    If $\lambda\in\R$ is such that for any of the pairs considered above, the inequality 
    \begin{align}
	   \label{eq:keyode}
	   -\partial_t\left(\delta\sum_\kappa\frc{(\welo_\kappa(s,t))}{\rho_\kappa(s,t)}\right) 
	   \ge \lambda\left(\delta\sum_\kappa\frc{(\welo_\kappa(s,t))}{\rho_\kappa(s,t)}\right) 
	   + 2\partial_s\fnc(\rho(s,t))
    \end{align}
    holds at every $s\in(0,1)$ and every $t\in(0,\tau)$, then \eqref{eq:formalgf} defines a $\lambda$-contractive gradient flow.
\end{lemma}
\begin{proof}[Proof of Lemma~\ref{lem:smooth}]
	We extend $(\rho^*,\welo^*)$ to arbitrary $s\in\R$ by defining $(\rho^*(s),\welo^*(s))=(\rho^*(0),0)$ for $s<0$, and $(\rho^*(s),\welo^*(s))=(\rho^*(1),0)$ for $s>1$, respectively. These extensions are in $H^1_\text{loc}(\R)\times L^2(\R)$, respectively, and the continuity equation \eqref{eq:dcont} then holds at almost any $s\in\R$.
	
	The approximations are now obtained by convolution of both $\rho^*$ and $\welo^*$. We combine two convolutions, one in (discrete) space to enforce positivity of $\rho$, and one in (continuous) time to enforce smoothness. For the convolution in space, we use an exponential kernel $K^{\delta,\eps}$ given by
	\begin{align*}
		K^{\delta,\eps}_\lambda = N^{\delta,\eps}\exp\big(-|\kappa\delta|/\eps\big), \quad \text{with}\quad N^{\delta,\eps} = \left(\frac2{1-\exp(-\delta/\eps)}-1\right)^{-1}.
	\end{align*}
	With the $\delta$-convolution from \eqref{eq:dconvolve}, we define 
	\begin{align*}
		\hat\rho^\eps(s)
		= K^{\delta,\eps}\ast_\delta \rho^*(s),
		\quad
		\hat\welo^\eps(s)
		= K^{\delta,\eps}\ast_\delta \welo^*(s).
	\end{align*}
	For convolution in time, we use compactly supported smooth non-negative and unit-mass mollifiers $\chi^\eps$ of the usual form $\chi^\eps(s) = \eps^{-1}\chi(\eps^{-1}s)$,
	\begin{align*}
		\rho^\eps_\kappa(s) = \chi^\eps\ast_s\hat\rho^\eps_\kappa(s) = \int_{\R} \chi^\eps(s-s')\hat\rho^\eps_\kappa(s')\dd s', 
		\quad
		\welo^\eps_\kappa(s) = \chi^\eps\ast_s\hat\rho^\eps_\kappa(s) = \int_{\R} \chi^\eps(s-s')\hat\rho^\eps_\kappa(s')\dd s'.
	\end{align*}
	It is easily seen that the continuity equation \eqref{eq:dcont} is preserved, and so are the total mass and the continuity of the second moment. $\rho^\eps$'s and $\welo^\eps$'s smoothness in time is an immediate consequence of $\chi^\eps$'s smoothness, $\rho^\eps$'s positivity is a consequence of $K^{\delta,\eps}$'s positivity. Weak convergence in $\curves$ follows by standard arguments for approximation of Sobolev functions by convolution, and since the discrete heat flow is uniformly Lipschitz continuous in $\delltwo$ up to the initial time. The bound  \eqref{eq:rhocompare} on the quotients of $\rho_\kappa$ follows from
	\begin{align*}
		\frac{\hat\rho^\eps_{\kappa+1}}{\hat\rho^\eps_\kappa}
		= \frac{\delta\sum_\lambda K^{\delta,\eps}_{\lambda+1}\rho^*_{\kappa-\lambda}}{\delta\sum_\lambda K^{\delta,\eps}_{\lambda}\rho^*_{\kappa-\lambda}}
		= \frac{\delta\sum_\lambda \frac{K^{\delta,\eps}_{\lambda+1}}{K^{\delta,\eps}_\lambda}K^{\delta,\eps}_\lambda\rho^*_{\kappa-\lambda}}{\delta\sum_\lambda K^{\delta,\eps}_\lambda\rho^*_{\kappa-\lambda}}
		\le \sup_\lambda\frac{K^{\delta,\eps}_{\lambda+1}}{K^{\delta,\eps}_\lambda}
		= \exp(\delta/\eps) =: C_\eps.
	\end{align*}
	The subsequent convolution in time preserves this inequality, so the bound is inherited by $\rho^\eps_{\kappa+1}/\rho^\eps_\kappa$. Finally, the approximation of the action follows by Jensen's inequality on the one hand,
	\begin{align*}
		\daction(\rho^\eps,\welo^\eps)
		&= \delta\sum_\kappa\int_0^1 \frac{[\chi^\eps\ast_t\hat\welo^\eps_\kappa](s)^2}{(\chi^\eps\ast_t\hat\rho^\eps_\kappa)(s)}\dd s 
		\le \delta\sum_\kappa\int_\R \frac{[\chi^\eps\ast_t\hat\welo^\eps_\kappa](s)^2}{(\chi^\eps\ast_t\hat\rho^\eps_\kappa)(s)}\dd s \\
		&\le \delta\sum_\kappa\int_\R \chi^\eps\ast_t\frc{(\hat\welo^\eps_\kappa)}{\hat\rho^\eps_\kappa}(s)\dd s 
		= \delta\sum_\kappa\int_\R \frac{[K^{\delta,\eps}\ast_\delta\welo^*(s)]_\kappa^2}{[K^{\delta,\eps}\ast_\delta\rho^*(s)]_\kappa}\dd s \\
		&\le \delta\sum_\kappa\int_\R K^{\delta,\eps}\ast_\delta\frc{\welo^*_\kappa(s)}{\rho^*_\kappa(s)}
		=\daction(\rho^*,\welo^*),
	\end{align*}
	and by lower semi-continuity of the action, see Lemma \ref{lem:lsc}, on the other hand. In combination, these force
	\begin{align*}
		\limsup_\eps \daction(\rho^\eps,\welo^\eps) \le \daction(\rho^*,\welo^*) \le \liminf_\eps \daction(\rho^\eps,\welo^\eps),
	\end{align*}
	and therefore convergence of $\daction(\rho^\eps,\welo^\eps)$ to $\daction(\rho^*,\welo^*)$.
\end{proof}

\begin{proof}[Proof of Lemma~\ref{lem:lambda-contractive}]
		We use the equivalence of the definition~\eqref{eq:evi} to the time-integrated version from \cite[Proposition 3.1]{DaneriSavare}, stating that a flow $\sgrp$ is $\lambda$-contractive with $\lambda<0$ iff 
	\begin{align}
		\label{eq:evii}
		\frac{e^{\lambda t}}2\dst\big(\sgrp^t(\bar\rho),\bar\eta\big)^2 - \frac12\dst\big(\bar\rho,\bar\eta\big)^2 
		\le \frac{e^{\lambda t}-1}{\lambda}\big[\fnc(\bar\eta)-\fnc\big(\sgrp^t(\bar\rho)\big)\big]
	\end{align}
	holds for all $\bar\rho$, $\bar\eta$ and all $t\ge0$.
	
	To verify~\eqref{eq:evii}, pick for given $\bar\rho$ and $\bar\eta$ a geodesic $(\tilde\rho,\tilde\welo)\in\curves(\bar\rho,\bar\eta)$, i.e., an action minimizer connecting $\tilde\rho(0)=\bar\eta$ to $\tilde\rho(1)=\bar\rho$, which exists by Lemma \ref{lem:geodesic}. For each $\eps>0$, consider a smooth approximation $(\tilde\rho^\eps,\tilde\welo^\eps)\in\curves$ according to Lemma \ref{lem:smooth}. Next, extend the single curve $(\tilde\rho^\eps,\tilde\welo^\eps)$ to an entire family $(\rho^\eps(s,t),\welo^\eps(s,t))$ with $s\in[0,1]$ and $t\ge0$ such that $(\rho^\eps(s,0),\welo^\eps(s,0))=(\tilde\rho^\eps(s),\tilde\welo^\eps(s))$, such that $(\rho^\eps(\cdot,t),\welo^\eps(\cdot,t))\in\curves$, and such that $\rho^\eps(s,t)=\sgrp^{st}(\rho^\eps(s))$ for all $s\in[0,1]$ and $t\ge0$. Note that $\rho^\eps(0,t)=\tilde\rho^\eps(0)$ is independent of $t$, and $\rho^\eps(1,t)=\sgrp^t(\tilde\rho^\eps(0))$ ``follows the flow''. One easily checks that this double-parametric family $(\rho^\eps(s,t),\welo^\eps(s,t))$ is admissible in Lemma \ref{lem:lambda-contractive}.

The task is then to show \eqref{eq:evii} with $\tilde\rho^\eps(1)$ in place of $\bar\rho$ and $\tilde\rho^\eps(0)$ in place of $\bar\eta$, which can equivalently be written as
\begin{align}
    \label{eq:eviii}
    \frac{e^{\lambda t}}2\dst\big(\rho^\eps(1,t),\rho^\eps(0,t)\big)^2 - \frac12\dst\big(\rho^\eps(1,0),\rho^\eps(0,0)\big)^2 
    \le \frac{e^{\lambda t}-1}{\lambda}\big[\fnc(\rho^\eps(0,t))-\fnc\big(\rho^\eps(1,t)\big)\big]
\end{align}
Once \eqref{eq:eviii} has been achieved, \eqref{eq:evii} follows easily by taking the limit $\eps\searrow0$, thanks to the properties of the $\eps$-approximation, see Lemma \ref{lem:smooth}, thanks to lower semi-continuity of $\dst$ and of $\fnc$, and thanks to the continuity of the semigroup. For proving \eqref{eq:eviii}, it suffices to show that
\begin{align}
    \label{eq:eviv}
    \frac{e^{\lambda t}}2\int_0^1\delta\sum_\kappa\frc{(\welo_\kappa^\eps(s,t))}{\rho^\eps_\kappa(s,t)}\dd s
    - \frac12 \int_0^1\delta\sum_\kappa\frc{(\welo^\eps_\kappa(s,0))}{\rho^\eps_\kappa(s,0)}\dd s
    \le \frac{e^{\lambda t}-1}{\lambda}\big[\fnc(\rho^\eps(0,t))-\fnc\big(\rho^\eps(1,t)\big)\big]
\end{align}
using the definition of $\dst$ as an infimum, and the fact that $(\rho^\eps(\cdot,0),\welo^\eps(\cdot,0)$ is a geodesic. By the energy-diminishing property of the gradient flow, $\fnc(\rho^\eps(s,t))$ is a non-increasing function of $t$. It is then elementary to deduce \eqref{eq:eviv} from the differential-in-$t$ version
\begin{align*}
    \frac12\partial_t\left(\int_0^1\delta\sum_\kappa\frc{(\welo^\eps_\kappa(s,t))}{\rho^\eps_\kappa(s,t)}\dd s\right) 
    +\frac\lambda 2\left(\int_0^1\delta\sum_\kappa\frc{(\welo^\eps_\kappa(s,t))}{\rho^\eps_\kappa(s,t)}\dd s\right) 
    \le \fnc(\rho^\eps(0,t))-\fnc\big(\rho^\eps(1,t)\big);
\end{align*}
simply multiply by $e^t$ and integrate with respect to $t$. Finally, since $(\rho^\eps(s,t),\welo^\eps(s,t))$ is smooth with respect to both $s\in(0,1)$ and $t>0$, and the curves $\rho^\eps(\cdot,t)$ are regular, it is clearly sufficient to verify the differential-in-$s$-and-$t$ variant~\eqref{eq:keyode}.
\end{proof}

\subsection{Inhomogeneous diffusion}
Consider the diffusive gradient flow from \eqref{eq:dFextern}, induced by the functional (note the sign)
\begin{align*}
    \fnc(\rho) = -\delta\sum_\kappa V^\delta_\kappa\rho_\kappa.
\end{align*}
Here $V^\delta:\Z\to\R$ can be understood as an approximation to an external potential $V\in C^4(\R)$ with the correspondence $V^\delta_\kappa\approx V(\kappa\delta)$ (cp. with Proposition~\ref{prp:formal:convexity}, part~\ref{itm:potconvexity}). Since
\[\ \frac{\partial\fnc}{\partial\rho_\kappa} \equiv -V^\delta_\kappa \]
independently of $\rho$, the associated gradient flow equation \eqref{eq:formalgf} amounts to
\begin{align}
    \label{eq:gf.1}
    \dot\rho_\kappa = \ddff_\kappa(\rho\,\ddff V^\delta).
\end{align}
We begin by proving that \eqref{eq:gf.1} indeed defines a semigroup. 
\begin{lemma}%
    \label{lem:semigroup}
    Let a bounded $A:\Z\to\R$ be given. Then the solution of the (infinite) system of ordinary differential equations
	\begin{align*}
		\dot\rho_\kappa = \ddff_\kappa(A\rho)
	\end{align*}
	defines a smooth semigroup on $\dprbtwo$.
\end{lemma}
\begin{proof}
    It suffices to verify that $\rho\mapsto\ddff(A\rho)$ is a globally Lipschitz continuous map from $\dprbtwo$ to $\delltwo$. Indeed, a rough estimate provides 
	\begin{multline*}
		[1+(\delta\kappa)^2]\big|\ddff_\kappa(A\rho)-\ddff_\kappa(A\eta)\big| \\
		\le \delta^{-2}[1+(\delta\kappa)^2]\big|A_{\kappa+1}(\rho_{\kappa+1}-\eta_{\kappa+1})+A_{\kappa-1}(\rho_{\kappa-1}-\eta_{\kappa+1})-2A_\kappa(\rho_\kappa-\eta_\kappa)\big| \\
		\shoveright{\le \frac{1+\delta^2}{\delta^2}\|A\|_{\ell^\infty}|\big([1+(\delta(\kappa+1))^2]|\rho_{\kappa+1}-\eta_{\kappa+1}|+[1+(\delta(\kappa-1))^2]|\rho_{\kappa-1}-\eta_{\kappa-1}|} \\ 
        +2[1+(\delta\kappa)^2]|\rho_\kappa-\eta_\kappa|\big),
	\end{multline*}
    where we have used the elementary bound
    \begin{align*}
        1+(\delta\kappa)^2
        \le (1+\delta^2)[1+\big(\delta(\kappa\pm1)\big)^2],
    \end{align*}
    for all $\kappa\in\Z$. We thus arrive at the Lipschitz estimate $\|\ddff(A\rho)-\ddff(A\eta)\|_{\delltwo}\le8\|A\|_{\ell^\infty}\delta^{-2}\|\rho-\eta\|_{\delltwo}$.
\end{proof}
\begin{proof}[Proof of Theorem~\ref{thm:contractivity}, part \ref{lbl:contractV}]
    We proceed as outlined in Section \ref{sct:lambdacontract} above. Let $\rho^\eps(s,t)$ be a smooth family of densities, subject to \eqref{eq:sideode}, which in this specific situation is
    \begin{align*}
        \partial_t\rho^\eps = s\ddff(\rho^\eps\,\ddff V^\delta),
        \quad
        \partial_t\welo^\eps = \ddff V^\delta\,\big(\rho^\eps+s\ddff\welo^\delta).
    \end{align*}
    We need to show \eqref{eq:keyode}, which is
    \begin{align}
        \label{eq:key.1}
        -\partial_t\left(\delta\sum_\kappa\frc{(\welo^\eps_\kappa(s,t))}{\rho^\eps_\kappa(s,t)}\right) 
        \ge \lambda\left(\delta\sum_\kappa\frc{(\welo^\eps_\kappa(s,t))}{\rho^\eps_\kappa(s,t)}\right) 
        - 2\delta\sum_\kappa \ddff_\kappa\welo^\eps(s,t)\,V^\delta_\kappa.
    \end{align}
    To enhance readability, we suppress the explicit dependence on $s$ and $t$ as well as the regularization parameter $\eps>0$ in the following. For a fixed integer $K\ge1$, we obtain    
    \begin{subequations}
        \begin{align}
            \nonumber
            -\partial_t\left(\delta\sum_{|\kappa|\le K} \frac{\welo^2}{\rho}\right)
            \nonumber
            &= \delta\sum_{|\kappa|\le K}\left[\left(\frac \welo\rho\right)^2\partial_t\rho - 2\frac\welo\rho \partial_t\welo\right] \\
            \nonumber
            &= s\,\delta\sum_{|\kappa|\le K}\left[
                \left(\frac \welo\rho\right)^2\ddff(\rho\ddff V^\delta)
                -2\ddff V^\delta\frac \welo\rho\ddff \welo
            \right] 
            -2\delta\sum_{|\kappa|\le K} \welo\ddff V^\delta \\
             \label{eq:evicorebc}
            &= s\bdry{\left(\frac \welo\rho\right)^2}{\rho\ddff V^\delta} 
            -2s\bdry{\ddff V^\delta\frac \welo\rho}{\welo}
            - 2\bdry{\welo}{V^\delta} \\ 
            \label{eq:evicorequad}
            &\qquad + s\,\delta\sum_{|\kappa|\le K}\left[
                \rho\ddff V^\delta\ddff\left(\left(\frac \welo\rho\right)^2\right)
                -2\welo\ddff\left(\ddff V^\delta\frac \welo\rho\right)
            \right]
            -2\delta\sum_{|\kappa|\le K} V^\delta\ddff \welo.
        \end{align}
    \end{subequations}
    The terms in \eqref{eq:evicorebc} and \eqref{eq:evicorequad} are now discussed individually. 
    Concerning \eqref{eq:evicorebc}: first, recall that by the assumed boundedness of $\ddff V^\delta$, there is a constant $C_V$ such that
    \begin{align}\label{eq:CV}
        0<\ddff_\kappa V^\delta\le C_V, \quad -C_V \le V^\delta_\kappa \le C_V\big(1+(\delta\kappa)^2\big).
    \end{align}
    In combination with the regularity estimate \eqref{eq:rhocompare}, this allows the following estimate on one of the four terms that contribute to $\bdry{\big(\frac \welo\rho\big)^2}{\rho\ddff V^\delta}$:
    \begin{align*}
        0\le\frac{\welo_K^2}{\rho_K^2}\rho_{K+1}\ddff_{K+1}V^\delta
        = \frac{\rho_{K+1}}{\rho_K}\frac{\welo_K^2}{\rho_K}\ddff_{K+1}V^\delta
        \le C_V C_\eps\frac{\welo_K^2}{\rho_K}.
    \end{align*}
    The upper bound is summable in $K$. Therefore, these expressions converges to zero as $K\to\infty$. The same is true for the other three terms contributing to $\bdry{\big(\frac \welo\rho\big)^2}{\rho\ddff V^\delta}$. A similar argument applies to $\bdry{\ddff V^\delta\frac \welo\rho}{\welo}$, where one of the four contributions can be estimated as follows:
    \begin{align*}
        \left|\ddff_K V^\delta\frac{\welo_K\welo_{K+1}}{\rho_K}\right|
        \le \ddff_K V^\delta\sqrt{\frac{\rho_{K+1}}{\rho_K}}\frac{\welo_K}{\sqrt{\rho_K}}\frac{\welo_{K+1}}{\sqrt{\rho_{K+1}}}
        \le \frac{C_V\sqrt{C_\eps}}2\left[\frac{\welo_K^2}{\rho_K}+\frac{\welo_{K+1}^2}{\rho_{K+1}}\right].
    \end{align*}
    Again, the bound on the right is summable in $K$, and thus forces the expression on the left to converge to zero for $K\to\infty$. A control on the last boundary term $\bdry{\welo}{V^\delta}$ is a bit more difficult to obtain: first, note that
    \begin{align*}
        \delta\sum_\kappa (1+|\delta\kappa|)|\welo_\kappa| 
        \le \left(\delta\sum_\kappa \frac{\welo_\kappa^2}{\rho_\kappa}\right)^{1/2}\left(\delta\sum_\kappa(1+|\delta\kappa|)^2\rho_\kappa\right)^{1/2},
    \end{align*}
    which is finite since the action and $\rho$'s second moment are finite. Therefore, $(1+|\delta\kappa|)|\welo_\kappa|$ is summable. Combine this with
    \begin{align*}
        \big(1+|\delta K|\big)^{-1}\big|\welo_KV^\delta_{K+1}\big|
        \le C_V\big(1+|\delta K|\big)|\welo_K|.
    \end{align*}
    The right-hand side above is summable, but $(1+|\delta K|)^{-1}$ on the left-hand side is \emph{not} summable. Therefore, there must exist a sequence $K_n\to\infty$ such that $\big|\welo_{K_n}V^\delta_{K_n+1}\big|\to0$. It is easily seen that $(K_n)_n$ can be chosen such that $\bdry[{K_n}]{\welo}{V^\delta}\to0$.

    Now for \eqref{eq:evicorequad}:
    the second sum therein converges to the second sum on the right-hand side of \eqref{eq:key.1} for $K_n\to\infty$, as desired. The first sum is estimated $\kappa$-wise, with the goal to verify the lower bound $\lambda\welo_\kappa^2/\rho_\kappa$. Expanding the two terms inside the sum by means of  the product rule \eqref{eq:product} for the discrete Laplacian, we obtain:
    \begin{align*}
        \rho\ddff V^\delta\ddff\left(\left(\frac \welo\rho\right)^2\right)
        &= \rho\left\{2\ddff V^\delta\frac \welo\rho\ddff\left(\frac \welo\rho\right)+\ddff V^\delta\left(\left[\delta_+\left(\frac \welo\rho\right)\right]^2+\left[\delta_-\left(\frac \welo\rho\right)\right]^2\right)\right\}, \\
        2\welo\ddff\left(\ddff V^\delta\frac \welo\rho\right)
        &= \rho\Bigg\{
        2(\ddff\ddff V^\delta)\left(\frac \welo\rho\right)^2 + 2\ddff V^\delta\,\frac \welo\rho\ddff\left(\frac \welo\rho\right) \\ 
        &\qquad +  2\left((\delta_+ \ddff V^\delta) \delta_+\left(\frac \welo\rho\right) + (\delta_-\ddff V^\delta)\delta_-\left(\frac \welo\rho\right)\right)\frac \welo\rho \Bigg\}.
    \end{align*}
    Consequently,
    \begin{align*}
        &\rho\ddff V^\delta\ddff\left(\left(\frac \welo\rho\right)^2\right) 
        -2\welo\ddff\left(\ddff V^\delta\frac \welo\rho\right) \\
        &\quad = \rho\left\{
        \ddff V^\delta\left[\delta_+\left(\frac \welo\rho\right)\right]^2
        -2(\delta_+\ddff V^\delta)\frac \welo\rho\delta_+\left(\frac \welo\rho\right)
        -(\ddff\ddff V^\delta)\left(\frac \welo\rho\right)^2
        \right\} \\
        &\qquad + \rho\left\{
        \ddff V^\delta\left[\delta_-\left(\frac \welo\rho\right)\right]^2
        -2(\delta_-\ddff V^\delta)\frac \welo\rho\delta_-\left(\frac \welo\rho\right)
        -(\ddff\ddff V^\delta)\left(\frac \welo\rho\right)^2
        \right\}.
    \end{align*}
    The expressions inside the curly brackets are two-homogeneous polynomials in the variables $\welo/\rho$ and $\delta_\pm(\welo/\rho)$. Generally, for given real coefficients $a$, $b$ and $c$ with $a>0$, one has
    \begin{align}
        \label{eq:elementary.1}
        au^2+2buv+cv^2 \ge \left(c-\frac{b^2}a\right)v^2
        \quad \text{for all real $u$ and $v$},
    \end{align}
    and therefore
    \begin{align*}
        \rho\ddff V^\delta\ddff\left(\left(\frac \welo\rho\right)^2\right) 
        -2\welo\ddff\left(\ddff V^\delta\frac \welo\rho\right)
        &\ge -\left(2\ddff\ddff V^\delta+\frac{(\delta_+\ddff V^\delta)^2+(\delta_-\ddff V^\delta)^2}{\ddff V^\delta}\right)\,\frac{\welo^2}\rho \\
        &\qquad = - \frac{\ddff\big[(\ddff V^\delta)^2\big]}{\ddff V^\delta}\,\frac{\welo^2}\rho .
    \end{align*}
    The quantity $\Lambda$ defined in \eqref{eq:lambda.1} is obviously an upper bound on the expression in round brackets on the right-hand side. We have thus established the validity of \eqref{eq:key.1}, finishing the proof.
\end{proof}

\begin{comment}

\begin{proposition}
	Assume that $\ddff V\ge0$. Then the gradient flow \eqref{eq:gf.1} is contractive in $\dhell$.
\end{proposition}
%
\begin{proof}
	We differentiate in time:
	\begin{align*}
		\frac12\frac{\dd}{\dn t}\left(\delta\sum_\kappa\big(\sqrt{\rho}-\sqrt{\eta}\big)^2\right)
		&= \delta\sum_\kappa\big(\sqrt{\rho}-\sqrt{\eta}\big)\Big(\frac{\dot\rho}{\sqrt{\rho}}-\frac{\dot\eta}{\sqrt{\eta}}\Big) \\
		&= \delta\sum_\kappa \left\{
		\left(\sqrt{\frac\eta\rho}-1\right)\ddff(\ddff V\,\rho) + 
		\left(\sqrt{\frac\rho\eta}-1\right)\ddff(\ddff V\,\eta)
		\right\} \\
		&= \delta\sum_\kappa \ddff V\left\{
		\rho\,\ddff\left(\sqrt{\frac\eta\rho}-1\right)
		+\eta\,\ddff\left(\sqrt{\frac\eta\rho}-1\right)
		\right\} \\
		&= \delta\sum_\kappa \ddff V\left\{
		\rho\left(\sqrt{\frac{\eta_+}{\rho_+}}+\sqrt{\frac{\eta_-}{\rho_-}}\right) 
		+\eta\left(\sqrt{\frac{\rho_+}{\eta_+}}+\sqrt{\frac{\rho_-}{\eta_-}}\right) 
		-4\sqrt{\rho\eta}
		\right\} \\
		&= \delta\sum_\kappa \ddff V\left\{
		\left(\sqrt[4]{\frac{\eta_-}{\rho_-}}\sqrt{\rho}-\sqrt[4]{\frac{\rho_-}{\eta_-}}\sqrt{\eta}\right)^2
		+ \left(\sqrt[4]{\frac{\eta_+}{\rho_+}}\sqrt{\rho}-\sqrt[4]{\frac{\rho_+}{\eta_+}}\sqrt{\eta}\right)^2
		\right\},
	\end{align*}
	which is non-positive since $\ddff V\le0$.
\end{proof}

\end{comment}

%

\subsection{Diffusion with convolution kernel}
We consider the functional (again, note the sign)
\begin{align*}
    \fnc(\rho) = -\frac{\delta}2\sum_\kappa \rho(W^\delta\ast_\delta\rho)
    = -\frac{\delta^2}2\sum_{\kappa,\lambda} W^\delta_{\kappa-\lambda}\rho_\kappa\rho_\lambda.
\end{align*}
Here $W^\delta:\Z\to\R$ should be understood as an approximation to an interaction kernel $W\in C^4(\R)$. Assuming symmetry, $W_{-\kappa}=W_\kappa$, we obtain
\begin{align}
    \label{eq:varyinteract}
    \frac{\partial\fnc}{\partial\rho_\kappa} = - \big[W^\delta\ast_\delta\rho\big]_\kappa = -\delta\sum_\lambda W^\delta_{\kappa-\lambda}\rho_\lambda.
\end{align}
Thus, the associated gradient flow equation amounts to
\begin{align}
    \label{eq:gf.2}
    \dot\rho_\kappa = \ddff\big(\rho\,\ddff\big[W^\delta\ast_\delta\rho]\big).
\end{align}
To show to \eqref{eq:gf.2} gives rise to a continuous semigroup on $\dprbtwo$, one proceeds in analogy to the proof of Lemma \ref{lem:semigroup}, deriving a $\delta$-dependent Lipschitz constant on the right-hand side. 
\begin{proof}[Proof of Theorem~\ref{thm:contractivity}, part \ref{lbl:contractW}]
    We proceed as in the proof for Theorem~\ref{thm:contractivity}, part \ref{lbl:contractV} above. The approximating family from~\eqref{eq:sideode} takes the form
    \begin{align*}
        \partial_t\rho^\eps = s\ddff(\rho^\eps\ddff W^\delta\ast_\delta\rho^\eps),
        \quad
        \partial_t\welo^\eps = 
        \rho^\eps\ddff W^\delta\ast_\delta\rho^\eps 
        + s\ddff\welo^\eps W^\delta\ast_\delta\rho^\eps
        + s \rho^\eps\ddff W^\delta\ast_\delta\ddff\welo^\eps. 
    \end{align*}
    We thus obtain, again omitting all unnecessary indices, and using that $\ddff(f\ast g)=\ddff f\ast g=f\ast\ddff g$:
    \begin{align}
        \MoveEqLeft\nonumber
        -\partial_t\bra[\bigg]{\delta\sum_\kappa\frac{\welo^2}{\rho}}
        = \delta\sum_\kappa\pra[\bigg]{\bra[\bigg]{\frac\welo\rho}^2\partial_t\rho-2\frac\welo\rho\partial_t\welo} \\
        \nonumber
        &=s\delta\sum_\kappa
            \pra[\bigg]{
                \left(\frac\welo\rho\right)^2\ddff(\rho\ddff W^\delta\ast\rho)
                -2\frac\welo\rho\ddff\welo\ddff W^\delta\ast\rho
                -2\welo\ddff W^\delta\ast\ddff\welo
            } \\
        \nonumber
        &\qquad
        -2\delta\sum_\kappa\welo\,\ddff W^\delta\ast\rho \\
        \label{eq:part.2a}
        &=s\delta\sum_\kappa
            \pra[\bigg]{
                \rho \ddff\bra[\bigg]{\left(\frac\welo\rho\right)^2}\ddff W^\delta\ast\rho
                -2\welo\ddff\left(\frac\welo\rho\ddff W^\delta\ast\rho\right)
                -2\welo\ddff\ddff W^\delta\ast\welo
            } \\            
        \label{eq:part.2b}
        &\qquad
        -2\delta\sum_\kappa \ddff\welo\,  W^\delta\ast\rho .
    \end{align}
    Above, we have omitted the boundary terms that arise in the summation by parts with the discrete Laplacian $\ddff$: these terms can be treated in full analogy to the previous proof, now with the convolution $\ddff W^\delta\ast\rho$ in place of the potential $V^\delta$. The relevant properties \eqref{eq:CV} is readily verified also for $\ddff W^\delta\ast\rho$.
    
    For the sum in \eqref{eq:part.2b}, we have
    \begin{align*}
        -2\delta\sum_\kappa \ddff\welo\,W^\delta\ast\rho 
        = -2\delta\sum_\kappa \partial_t\rho\,W^\delta\ast\rho
        = -\partial_t\bra[\bigg]{\delta\sum_\kappa\rho\,W^\delta\ast\rho}
        = -2\partial_t\fnc(\rho).
    \end{align*}
    Concerning the first two terms inside the sum in line \eqref{eq:part.2a} above:
    \begin{equation}
        \label{eq:techi}
        \begin{split}            
        \rho \ddff\set[\bigg]{\bra[\bigg]{\frac\welo\rho}^{\!2}}\ddff W^\delta\ast\rho
        &= \rho\Biggl\{
            2(\ddff W^\delta\ast\rho)\,\frac\welo\rho\ddff\frac\welo\rho 
           +(\ddff W^\delta\ast\rho)\,\pra[\bigg]{\bra[\bigg]{\delta_+\frac\welo\rho}^{\!2}+\bra[\bigg]{\delta_-\frac\welo\rho}^{\!2}}\Biggr\}, \\
        2\welo\ddff\left\{\frac\welo\rho\ddff W^\delta\ast\rho\right\}
        &= \rho\,\Biggl\{
            2(\ddff\ddff W^\delta\ast\rho)\,\left(\frac\welo\rho\right)^{\!2}
            +2(\ddff W^\delta\ast\rho)\,\frac\welo\rho\ddff\frac\welo\rho \\
            &\qquad\quad
            +2\left[
                (\delta_+\ddff W^\delta\ast\rho)\,\delta_+\frac\welo\rho
                +(\delta_-\ddff W^\delta\ast\rho)\,\delta_-\frac\welo\rho
            \right]\frac\welo\rho
        \Biggr\}.
        \end{split}
    \end{equation}
    At this point, we need an elementary inequality that is analogous to \eqref{eq:elementary.1} in the proof for potentials of Theorem~\ref{thm:contractivity}~\eqref{itm:potconvexity}: given three bounded functions $A$, $B$, $C$ on $\Z$ of which $A$ is everywhere positive, then for all real $u$ and $v$,
    \begin{align}
        \label{eq:elementary.2}
        (A\ast\rho)u^2 + 2(B\ast\rho) uv + (C\ast\rho)v^2 
        = (u^2A+2uvB+v^2C)\ast\rho
        \ge \left[\left(C-\frac{B^2}A\right)\ast\rho\right]v^2.
    \end{align}
    Substitute \eqref{eq:elementary.2} with the obvious choices for $A$, $B$, $C$ and $u$, $v$ into \eqref{eq:techi}, and recall that $\rho$ is non-negative and of unit mass to obtain
    \begin{align*}
        \MoveEqLeft
        \rho \ddff\set[\bigg]{\bra[\bigg]{\frac\welo\rho}^{\!2}}\ddff W^\delta\ast\rho
        -2\welo\ddff\left\{\frac\welo\rho\ddff W^\delta\ast\rho\right\} \\
        &= \rho \left\{
            (\ddff W^\delta\ast\rho)\,\left(\delta_+\frac\welo\rho\right)^{\!2}
            -2(\delta_+\ddff W^\delta\ast\rho)\,\delta_+\frac\welo\rho
            -(\ddff\ddff W^\delta\ast\rho)\,\left(\frac\welo\rho\right)^{\!2}
        \right\} \\
        &\qquad + \rho \left\{
            (\ddff W^\delta\ast\rho)\,\left(\delta_- \frac\welo\rho\right)^{\!2}
            -2(\delta_-\ddff W^\delta\ast\rho)\,\delta_- \frac\welo\rho
            -(\ddff\ddff W^\delta\ast\rho)\,\left(\frac\welo\rho\right)^{\!2}
        \right\} \\
        &\ge -\rho\left\{
            \left(\ddff\ddff W^\delta+\frac{(\delta_+\ddff W^\delta)^{2}}{\ddff W^\delta}\right)\ast\rho
            +\left(\ddff\ddff W^\delta+\frac{(\delta_-\ddff W^\delta)^{2}}{\ddff W^\delta}\right)\ast\rho
        \right\}\left(\frac\welo\rho\right)^{\!2}\\
        &\ge -\sup\left(2\ddff\ddff W^\delta+\frac{(\delta_+\ddff W^\delta)^2+(\delta_-\ddff W^\delta)^2}{\ddff W^\delta}\right)\,\frac{\welo^2}\rho
        = -\sup\left(\frac{\ddff\big[(\ddff W^\delta)^2\big]}{\ddff W^\delta}\right)\,\frac{\welo^2}\rho.
    \end{align*}
    Finally, we obtain for the last term in the sum in line \eqref{eq:part.2a} by the Cauchy-Schwarz inequality that
    \begin{align*}
        &-\delta\sum_\kappa\welo_\kappa\big[\ddff\ddff W^\delta \ast\welo\big]_\kappa
        = -\delta^2\sum_{\kappa,\lambda} \welo_\kappa\big[\ddff\ddff W^\delta\big]_{\kappa-\lambda}\welo_\lambda \\
        &\qquad \ge -\sup|\ddff\ddff W^\delta|\,\left(\delta\sum_\kappa|\welo_\kappa|\right)^2 
        \ge -\sup|\ddff\ddff W^\delta|\,\left(\delta\sum\rho\right)\left(\delta\sum\frac{\welo^2}\rho\right).
    \end{align*}
    In combination, we arrive at \eqref{eq:keyode} with $\lambda=-\Lambda$ for $\Lambda$ given in \eqref{eq:lambda.2}.
\end{proof}

\subsection{Quadratic porous medium equation}
Finally, we consider the special interaction energy functional
\begin{align*}
    \fnc(\rho) = -\frac{\delta}4\sum_{\kappa,\lambda} \delta|\kappa-\lambda|\rho_\kappa\rho_\lambda. 
\end{align*}
Note that the choice $W^\delta_\kappa=|\kappa|$ does not fit the previous case since $W^\delta$ lacks strict convexity. Still, the variation can be computed as in \eqref{eq:varyinteract},
\begin{align*}
    \frac{\partial\fnc}{\partial\rho_\kappa} 
    = -\frac{\delta}2\sum_\lambda \delta|\kappa-\lambda|\rho_\lambda,
\end{align*}
and so
\begin{align*}
    \ddff_\kappa\frac{\partial\fnc}{\partial\rho} 
    = -\frac{\delta}2\sum_\lambda \delta|\lambda-\kappa|\ddff_\lambda\rho
    = -\delta\sum_{\lambda>\kappa}\delta(\lambda-\kappa),\ddff_\lambda\rho,
\end{align*}
where the last equality follows since $\ddff\rho$ is of zero average. Now recalling Lemma \ref{lem:dLaplace}, we finally obtain that
\begin{align*}
    \ddff\frac{\partial\fnc}{\partial\rho} = (\ddff)^{-1}\ddff\rho = \rho,
\end{align*}
showing that the corresponding gradient flow is indeed given by the quadratic porous medium equation \eqref{eq:dFpme}, i.e., $\dot\rho=\ddff(\rho^2)$. Again, the semigroup property is verified in analogy to the proof of Lemma \ref{lem:semigroup}, using a $\delta$-dependent Lipschitz estimate of $\rho\mapsto\ddff\rho$ in $\dprbtwo$; recall that $\rho$ is uniformly bounded by $\rho\le1/\delta$.
\begin{proof}[Proof of Theorem~\ref{thm:contractivity}, part \ref{lbl:contractPME}]
    The approximating family has the form
    \begin{align*}
        \partial_t\rho^\eps = s\ddff\big[(\rho^\eps)^2\big],
        \quad
        \partial_t\welo^\eps = 
        (\rho^\eps)^2 + 2s\rho^\eps\ddff\welo^\eps.
    \end{align*}
    Neglecting super-indices, we obtain
    \begin{align*}
        -\partial_t\left(\delta\sum\frac{\welo^2}{\rho}\right)
        &= \delta\sum\left[\left(\frac\welo\rho\right)^2\partial_t\rho-2\frac\welo\rho\partial_t\welo\right] \\
        &= s\delta\sum\left[\left(\frac\welo\rho\right)^2\ddff(\rho^2)-4\welo\ddff\welo\right]
        -2\delta\sum\welo\rho \\
        &= s\delta\sum\left[4(\welo_{\kappa+1}-\welo_\kappa)^2 - \left(\frac{\welo_{\kappa+1}^2}{\rho_{\kappa+1}^2}-\frac{\welo_{\kappa}^2}{\rho_{\kappa}^2}\right)(\rho_{\kappa+1}^2-\rho_\kappa^2)\right] 
        -2\delta\sum\welo\rho. 
    \end{align*}
    On the one hand,
    \begin{align*}
        &4(\welo_{\kappa+1}-\welo_\kappa)^2 - \left(\frac{\welo_{\kappa+1}^2}{\rho_{\kappa+1}^2}-\frac{\welo_{\kappa}^2}{\rho_{\kappa}^2}\right)(\rho_{\kappa+1}^2-\rho_\kappa^2) \\
        &= 3(\welo_{\kappa+1}-\welo_\kappa)^2-2\welo_{\kappa+1}\welo_\kappa + \frac{\rho_\kappa^2}{\rho_{\kappa+1}^2}\welo_{\kappa+1}^2 + \frac{\rho_{\kappa+1}^2}{\rho_{\kappa}^2}\welo_{\kappa}^2 \\
        &= 3(\welo_{\kappa+1}-\welo_\kappa)^2 + \left(\frac{\rho_\kappa}{\rho_{\kappa+1}}\welo_{\kappa+1} - \frac{\rho_{\kappa+1}}{\rho_{\kappa}}\welo_{\kappa}\right)^2 \ge 0.
    \end{align*}
    And on the other hand, since the discrete Laplacian is self-adjoint,
    \begin{equation*}
        \delta\sum\welo\,\rho  
        = \delta\sum\welo\,\ddff(\ddff)^{-1}\rho
        = \delta\sum\ddff\welo\,(\ddff)^{-1}\rho
        = \delta\sum\partial_s\rho\,(\ddff)^{-1}\rho
        = \partial_s\fnc(\rho).      %
    \end{equation*}
    The boundary terms $\bdry{\welo}{(\ddff)^{-1}\rho}$ arising from the summation by parts with $\ddff$ are treated in a similar way as before.
\end{proof}

\bibliographystyle{abbrv}
\bibliography{main}	

\begin{thebibliography}{10}

\bibitem{AmbrosioButtazzo88}
L.~Ambrosio and G.~Buttazzo.
\newblock Weak lower semicontinuous envelope of functionals defined on a space
  of measures.
\newblock {\em Ann. Mat. Pura Appl. (4)}, 150:311--339, 1988.

\bibitem{AGS}
L.~Ambrosio, N.~Gigli, and G.~Savar{\'e}.
\newblock {\em Gradient flows: in metric spaces and in the space of probability
  measures}.
\newblock Springer Science \& Business Media, 2005.

\bibitem{Bleher}
P.~M. Bleher, J.~L. Lebowitz, and E.~R. Speer.
\newblock Existence and positivity of solutions of a fourth-order nonlinear
  {PDE} describing interface fluctuations.
\newblock {\em Communications on Pure and Applied Mathematics}, 47(7):923--942,
  1994.

\bibitem{Brenier2020hiddenconvexity}
Y.~Brenier.
\newblock {Examples of Hidden Convexity in Nonlinear PDEs}.
\newblock Lecture: hal-02928398, Sept. 2020.

\bibitem{CLSS}
J.~A. Carrillo, S.~Lisini, G.~Savar\'{e}, and D.~Slep\v{c}ev.
\newblock Nonlinear mobility continuity equations and generalized displacement
  convexity.
\newblock {\em J. Funct. Anal.}, 258(4):1273--1309, 2010.

\bibitem{CarMcCVil}
J.~A. Carrillo, R.~J. McCann, and C.~Villani.
\newblock Kinetic equilibration rates for granular media and related equations:
  entropy dissipation and mass transportation estimates.
\newblock {\em Rev. Mat. Iberoamericana}, 19(3):971--1018, 2003.

\bibitem{Chow}
S.-N. Chow, W.~Huang, Y.~Li, and H.~Zhou.
\newblock Fokker--{P}lanck equations for a free energy functional or {M}arkov
  process on a graph.
\newblock {\em Archive for Rational Mechanics and Analysis}, 203:969--1008,
  2012.

\bibitem{Cohen2024}
D.~W. Cohen.
\newblock A formal gradient flow interpretation of a class of {McKean}-vlasov
  equations using a new adaptation of the wasserstein metric, Nov. 2024.
\newblock {ZSCC}: 0000000.

\bibitem{CoxObloj2008}
A.~M.~G. Cox and J.~Ob\l~\'{o}j.
\newblock Classes of measures which can be embedded in the simple symmetric
  random walk.
\newblock {\em Electron. J. Probab.}, 13:no. 42, 1203--1228, 2008.

\bibitem{DaneriSavare}
S.~Daneri and G.~Savar\'{e}.
\newblock Eulerian calculus for the displacement convexity in the {W}asserstein
  distance.
\newblock {\em SIAM J. Math. Anal.}, 40(3):1104--1122, 2008.

\bibitem{DerridaLebowitzSpeerSpohn1991a}
B.~Derrida, J.~L. Lebowitz, E.~R. Speer, and H.~Spohn.
\newblock Dynamics of an anchored {T}oom interface.
\newblock {\em Journal of Physics A: Mathematical and General}, 24(20):4805,
  oct 1991.

\bibitem{DerridaLebowitzSpeerSpohn1991b}
B.~Derrida, J.~L. Lebowitz, E.~R. Speer, and H.~Spohn.
\newblock Fluctuations of a stationary nonequilibrium interface.
\newblock {\em Phys. Rev. Lett.}, 67:165--168, Jul 1991.

\bibitem{DNS}
J.~Dolbeault, B.~Nazaret, and G.~Savar\'{e}.
\newblock A new class of transport distances between measures.
\newblock {\em Calc. Var. Partial Differential Equations}, 34(2):193--231,
  2009.

\bibitem{EFS19}
M.~Erbar, M.~Fathi, and A.~Schlichting.
\newblock Entropic curvature and convergence to equilibrium for mean-field
  dynamics on discrete spaces.
\newblock {\em ALEA Lat. Am. J. Probab. Math. Stat.}, 17(1):445--471, 2020.

\bibitem{ErbarMaas2012}
M.~Erbar and J.~Maas.
\newblock Ricci curvature of finite markov chains via convexity of the entropy.
\newblock {\em Archive for Rational Mechanics and Analysis}, 206(3):997–1038,
  Aug. 2012.

\bibitem{ErbarMaas}
M.~Erbar and J.~Maas.
\newblock Gradient flow structures for discrete porous medium equations.
\newblock {\em Discrete Contin. Dyn. Syst.}, 34(4):1355--1374, 2014.

\bibitem{GianazzaSavareToscani2009}
U.~Gianazza, G.~Savar{\'e}, and G.~Toscani.
\newblock The {W}asserstein gradient flow of the {F}isher information and the
  quantum drift-diffusion equation.
\newblock {\em Archive for rational mechanics and analysis}, 194(1):133--220,
  2009.

\bibitem{GigliMaas}
N.~Gigli and J.~Maas.
\newblock Gromov-{H}ausdorff convergence of discrete transportation metrics.
\newblock {\em SIAM J. Math. Anal.}, 45(2):879--899, 2013.

\bibitem{HeHuOblojZhou2019}
X.~D. He, S.~Hu, J.~Ob\l~\'{o}j, and X.~Y. Zhou.
\newblock Two explicit {S}korokhod embeddings for simple symmetric random walk.
\newblock {\em Stochastic Process. Appl.}, 129(9):3431--3445, 2019.

\bibitem{HuesmannTrevisan2019}
M.~Huesmann and D.~Trevisan.
\newblock A {B}enamou-{B}renier formulation of martingale optimal transport.
\newblock {\em Bernoulli}, 25(4A):2729--2757, 2019.

\bibitem{JKO1998}
R.~Jordan, D.~Kinderlehrer, and F.~Otto.
\newblock The {{Variational Formulation}} of the {{Fokker-Planck Equation}}.
\newblock {\em SIAM J. Math. Anal.}, 29(1):1, 1998.

\bibitem{JunPin-analysis}
A.~J\"{u}ngel and R.~Pinnau.
\newblock Global nonnegative solutions of a nonlinear fourth-order parabolic
  equation for quantum systems.
\newblock {\em SIAM J. Math. Anal.}, 32(4):760--777, 2000.

\bibitem{Maas2011}
J.~Maas.
\newblock Gradient flows of the entropy for finite {M}arkov chains.
\newblock {\em J. Funct. Anal.}, 261(8):2250--2292, 2011.

\bibitem{MaasMatthes}
J.~Maas and D.~Matthes.
\newblock Long-time behavior of a finite volume discretization for a fourth
  order diffusion equation.
\newblock {\em Nonlinearity}, 29(7):1992, 2016.

\bibitem{MRSS}
D.~Matthes, E.-M. Rott, G.~Savaré, and A.~Schlichting.
\newblock A structure preserving discretization for the
  {D}errida-{L}ebowitz-{S}peer-{S}pohn equation based on diffusive transport.
\newblock {\em arXiv preprint:2312.13284}, 2023.

\bibitem{McCann}
R.~J. McCann.
\newblock A convexity principle for interacting gases.
\newblock {\em Adv. Math.}, 128(1):153--179, 1997.

\bibitem{MielkeMC}
A.~Mielke.
\newblock Geodesic convexity of the relative entropy in reversible {M}arkov
  chains.
\newblock {\em Calculus of Variations and Partial Differential Equations},
  48:1--31, 2013.

\bibitem{MisiatsYip}
O.~Misiats and N.~K. Yip.
\newblock Convergence of space-time discrete threshold dynamics to anisotropic
  motion by mean curvature.
\newblock {\em Discrete Contin. Dyn. Syst.}, 36(11):6379--6411, 2016.

\bibitem{Nika2023}
G.~Nika.
\newblock A gradient system for a higher-gradient generalization of {F}ourier's
  law of heat conduction.
\newblock {\em Modern Physics Letters B}, 37(11), Mar. 2023.

\bibitem{OttoPME}
F.~Otto.
\newblock The geometry of dissipative evolution equations: the porous medium
  equation.
\newblock {\em Comm. Partial Differential Equations}, 26(1-2):101--174, 2001.

\bibitem{OttoWest}
F.~Otto and M.~Westdickenberg.
\newblock Eulerian calculus for the contraction in the {W}asserstein distance.
\newblock {\em SIAM J. Math. Anal.}, 37(4):1227--1255, 2005.

\bibitem{PanXuLouYao2006}
L.~S. Pan, D.~Xu, J.~Lou, and Q.~Yao.
\newblock A generalized heat conduction model in rarefied gas.
\newblock {\em Europhysics Letters ({EPL})}, 73(6):846--850, Mar. 2006.

\end{thebibliography}

\end{document}